\documentclass[final]{siamart}

\usepackage{amsfonts,amssymb}
\usepackage{mathrsfs,fancybox,pifont}
\usepackage{graphicx, float}
\graphicspath{{./}}
\IfFileExists{placeins.sty}{\usepackage[section]{placeins}}{}
\usepackage{booktabs,multirow}
\usepackage{colortbl}
\usepackage{enumitem}
\usepackage{algpseudocode}
\usepackage{url}
\usepackage{color}
\usepackage{cases}
\usepackage{mathtools}
\usepackage{version}
\usepackage{subcaption}
\usepackage{bm}

\newsiamthm{assumption}{Assumption}
\newsiamremark{remark}{Remark}
\numberwithin{theorem}{section}
\theoremstyle{plain}
\theoremheaderfont{\normalfont\itshape}
\theorembodyfont{\normalfont}
\theoremseparator{.}
\theoremsymbol{}
\newtheorem{example}{Example}

\allowdisplaybreaks


\newcommand{\TheTitle}{Asymptotically Compatible Variational Energy Dissipation Laws and Structure-Preserving Variable-Step Linear Relaxation Schemes for a Subdiffusive Tumor--Nutrient Model}
\newcommand{\TheAuthors}{Hui Yu, Zhaoyang Wang, and Ping Lin}
\headers{Subdiffusive Tumor--Nutrient Relaxation Schemes}{H. Yu, Z. Wang, and P. Lin}
\title{\TheTitle}
\author{Hui Yu\thanks{School of Mathematics and Physics, University of Science and Technology Beijing, Beijing, 100083, China (\email{h.yu@ustb.edu.cn}).}
\and Zhaoyang Wang\thanks{Corresponding author. Department of Mathematics, Beijing Normal University, Zhuhai, Guangdong 519087, China (\email{zhaoyang584520@163.com}). The work of this author was supported by the National Natural Science Foundation of China (Grant No.~12501535) and the Scientific Research Fund of Beijing Normal University (Grant No.~12900-312200502405).}
\and Ping Lin\thanks{Corresponding author. Division of Mathematics, University of Dundee, Dundee, DD1 4HN, UK (\email{P.Lin@dundee.ac.uk}). The work of this author was supported by the National Natural Science Foundation of China (Grant No.~12371388) and the Beijing Natural Science Foundation (Grant No.~IS25005).}}

\ifpdf
\hypersetup{pdftitle={\TheTitle},pdfauthor={\TheAuthors}}
\fi

\begin{document}

\maketitle

\begin{abstract}
In this work, we develop a mass-conservative subdiffusive tumor--nutrient model coupled through chemotaxis and memory-dependent reaction exchange, 
together with structure-preserving variable-step linear relaxation schemes.  
A weighted Riemann--Liouville identity is proved and used to construct a variational energy and the associated energy dissipation law for the proposed model.
Both the model and the associated variational energy dissipation law are asymptotically compatible with their integer-order counterparts 
as the fractional order approaches one. 
For numerical approximation, we discretize the Riemann--Liouville fractional derivatives using the $L1_R$ formula on nonuniform time meshes and 
derive the associated discrete weighted inequality.
We first construct a variable-step linear relaxation scheme for the subdiffusive model.
Then, to mitigate alternating perturbations induced by the auxiliary-variable error and amplified by the $L1_R$ convolution at small fractional orders, 
we construct a temporally damped linear relaxation scheme. 
Since temporal damping leaves the auxiliary-variable recurrence unchanged, 
we further introduce an a posteriori relaxation correction to reduce alternating errors in the auxiliary variable.
All three schemes preserve total mass, are proven to satisfy discrete variational energy dissipation laws using the discrete weighted inequality, 
and remain consistent with the integer-order limit. 
Several numerical examples are presented to validate the accuracy and stability of the proposed schemes.  
\end{abstract}



\section{Introduction}

Tumor growth is driven by the coupled effects of cell proliferation, nutrient transport, chemotactic migration, 
and spatial heterogeneity in the microenvironment.  
As the tumor grows, the combined effects of nutrient diffusion and cellular consumption can generate spatial gradients in nutrient concentration, 
contributing to spatially heterogeneous tumor growth~\cite{demicco2024metabolic}.
Such nonuniform growth can lead to complex tumor morphologies, including protrusive and finger-like structures~\cite{frieboes2006integrated}.
Mathematical models provide a framework for investigating how interactions between tumor cells and their microenvironment influence tumor growth~\cite{oden2010general}.
 
Phase-field models describe tumor cells, nutrients, angiogenic factors, extracellular matrix, and host cells as interacting constituents represented by volume fractions or concentrations~\cite{xu2016mathematical,lima2014hybrid}.
For a closed tumor--nutrient system, the fluxes and reaction exchange must be compatible with total mass conservation and free-energy dissipation~\cite{wang2025stability}.
In such a closed system, classical diffusion may fail to describe anomalous transport in heterogeneous tumor microenvironments~\cite{jiang2014anomalous}.
Inspired by the work of Fritz et al.~\cite{fritz2021subdiffusive}, we introduce Riemann--Liouville fractional derivatives into both mass fluxes and reaction terms to formulate a closed subdiffusive tumor--nutrient model.
Our aim is to construct a variational energy for this system and ensure that both the model and its energy dissipation law are asymptotically compatible with their integer-order counterparts.
However, the standard Riemann--Liouville energy inequality does not directly accommodate the state-dependent reaction coefficient within the memory operator.
To address this difficulty, we derive a weighted Riemann--Liouville identity to establish a variational energy dissipation law for the subdiffusive mass-conservative coupled system.
 
The nonlocality of time-fractional derivatives makes it difficult to establish dissipation of the original energy in phase-field equations, and the corresponding energy properties are often characterized by energy boundedness~\cite{quan2023decreasing} or modified energy dissipation laws~\cite{tang2019energy,du2020time,quan2020define,qi2024unified}.
Numerical methods, such as the scalar auxiliary variable (SAV)~\cite{shen2018scalar,shen2019new} and its extensions~\cite{liu2020exponential,huang2020highly}, allow linear treatment of nonlinear potentials while maintaining energy stability.
Relaxation SAV methods further reduce errors associated with the ODE-based update of the auxiliary variable~\cite{jiang2022improving}.
Recent linear relaxation strategies compute the auxiliary variable from a linear algebraic equation,
thereby avoiding error accumulation from numerically solving the auxiliary-variable ODE~\cite{zhang2024linear,yu2026linear}.
However, the auxiliary-variable update in the linear relaxation scheme admits an undamped alternating error mode that can be amplified at small fractional orders in the history convolution arising from the $L1_R$ discretization of the Riemann--Liouville fractional derivatives.
To limit the propagation of this alternating error into the tumor phase variable, we introduce a temporal damping term (see Section~\ref{subsec:energy-compatible-damped-relaxation}).
Since this damping leaves the auxiliary-variable recurrence unchanged, we augment the damped scheme with an a posteriori relaxation correction that directly controls the amplitude of the alternating auxiliary-variable error passed to subsequent time steps (see Section~\ref{subsec:energy-compatible-damped-modified-relaxation}).
All three schemes conserve total mass and satisfy discrete variational energy dissipation laws, with both the schemes and their dissipation laws being asymptotically compatible with their integer-order counterparts.

This paper is organized as follows.
Section \ref{sec:modeling} introduces the subdiffusive tumor growth model coupled with nutrient dynamics and establishes the associated variational energy dissipation law.
Section \ref{sec:numericalscheme} develops the linear relaxation, damped linear relaxation, and damped modified linear relaxation schemes on nonuniform time meshes, 
establishes their discrete mass conservation and variational energy dissipation laws, and shows that both the schemes and their dissipation laws are asymptotically compatible with their integer-order counterparts.
Section \ref{sec:example} presents numerical experiments to verify the accuracy and structure-preserving properties of the schemes.

\section{Mathematical Modeling} \label{sec:modeling}
In this section, we derive a mass-conservative subdiffusive tumor growth model coupled with nutrient dynamics.
We then establish a weighted Riemann--Liouville identity and use it to construct a variational energy and establish the corresponding dissipation law, 
which is asymptotically compatible with its integer-order counterpart.

\subsection{Model Formulation}
Let $\Omega\subset\mathbb{R}^d$ ($d\in\{2,3\}$) be the spatial domain.
The order parameter $\phi(\bm{x},t)\in[0,1]$ denotes the tumor volume fraction ($\phi=1$ for tumor tissue and $\phi=0$ for healthy tissue), while $\sigma(\bm{x},t)$ denotes the nutrient concentration. 
Under the assumption of constant tumor-cell density, the local mass balance reads
\begin{equation}\label{masstotal}
    \frac{\partial \phi}{\partial t}= -\nabla \cdot \boldsymbol{J}_\phi+S_\phi, \quad \frac{\partial \sigma}{\partial t}=-\nabla \cdot \boldsymbol{J}_\sigma+S_\sigma.
\end{equation}
$\boldsymbol{J}_\phi$ and $\boldsymbol{J}_\sigma$ denote the tumor and nutrient mass fluxes, respectively, and $S_\phi,\ S_\sigma$ are the corresponding reaction terms.
$S_\phi = -S_\sigma$ is assumed to ensure total mass conservation in a closed system.

We assume that tumor dynamics are driven by a Ginzburg--Landau free energy incorporating chemotactic coupling between the tumor and nutrient.
The free energy is defined as~\cite{hawkins2012numerical}
\begin{equation}\label{eq:energy-free}
    E(\phi, \sigma)= \int_\Omega \left(
    \frac{\epsilon_\phi^2}{2}|\nabla \phi|^2
    + F(\phi) + \frac{1}{2\delta}\sigma^2
    + \chi(\phi, \sigma) \right) d\bm{x}.
\end{equation}
The first term in \eqref{eq:energy-free} represents the interfacial energy, with $\epsilon_\phi$ controlling the interface thickness.
The bulk double-well potential $F(\phi)=\frac{1}{4}\phi^2(1-\phi)^2$ drives phase separation.
The term $\frac{1}{2\delta}\sigma^2$ accounts for the energetic contribution of the nutrient-rich extracellular water, with the small parameter $\delta>0$ controlling the strength of this quadratic nutrient penalty, while $\chi(\phi,\sigma)$ represents the chemotactic energy.
Thus, the associated chemical potentials are
\begin{equation}\notag
    \mu_\phi = \frac{\delta E}{\delta \phi}
    = -\epsilon_\phi^2 \Delta \phi + f(\phi) + \partial_\phi \chi,
    \quad
    \mu_\sigma = \frac{\delta E}{\delta \sigma}
    = \frac{1}{\delta}\sigma + \partial_\sigma \chi,
\end{equation}
where $f(\phi)=F'(\phi)$.

To enforce the energy dissipation law $\frac{d}{dt}E(\phi, \sigma) \leq 0$, combining the local mass balances \eqref{masstotal} with the free energy \eqref{eq:energy-free} leads to the following constitutive choices (see~\cite{wang2025stability} for details):
\begin{equation}\notag
    \bm{J}_\phi = -M_\phi \nabla\mu_\phi, \quad \bm{J}_\sigma = -M_\sigma \nabla \mu_\sigma,\quad S_\phi = - S_\sigma = P(\phi)\left(\mu_\sigma - \mu_\phi \right),
\end{equation}
where $M_\phi$ and $M_\sigma$ are positive constant mobility coefficients for tumor and nutrient, respectively, and $P(\phi)$ is a nonnegative proliferation function.

We now extend these constitutive relations to incorporate memory effects associated with subdiffusion, an anomalous transport process characterized by a mean squared displacement $\langle x^2(t)\rangle\propto t^\alpha$ with $0<\alpha<1$.
A standard microscopic description is provided by continuous-time random walks (CTRWs), in which stochastic jumps are separated by random waiting times~\cite{montroll1965random,metzler2000random}.
For jump distributions with finite second moments, a waiting-time density with a power-law tail proportional to $t^{-1-\alpha}$ leads to a time-fractional diffusion equation in the continuum limit~\cite{compte1996stochastic,metzler2000random}.
CTRW-based models have been used to describe tumor-cell invasion by coupling anomalous migration with proliferative dynamics~\cite{fedotov2007migration}, and experiments have reported subdiffusive cancer-cell migration under specific microenvironmental conditions~\cite{sharma2026glassy}.
Reaction--subdiffusion theory further shows that the form of the memory terms depends on how reaction kinetics are coupled to the waiting-time process, with fractional memory potentially entering both transport and reaction terms~\cite{henry2006anomalous,nepomnyashchy2016mathematical}.
Motivated by~\cite{fritz2021subdiffusive}, we introduce Riemann--Liouville memory into the tumor and nutrient fluxes and their reaction exchange.
We use the same fractional order $\alpha$ for both fluxes and the reaction exchange, retaining $S_\phi=-S_\sigma$:
\begin{equation}\notag 
        \bm{J}_\phi = - M_\phi {^R\partial_t^{1-\alpha}} \nabla \mu_\phi,\quad  \bm{J}_\sigma = - M_\sigma {^R\partial_t^{1-\alpha}} \nabla \mu_\sigma,\quad
        S_\phi = {^R\partial_t^{1-\alpha}} \left( P(\phi) \left(\mu_\sigma - \mu_\phi \right) \right) , 
\end{equation}
where ${}^R\partial_t^{1-\alpha}$ denotes the Riemann--Liouville time-fractional derivative defined as
\begin{equation}\notag
    {}^R\partial_t^{1-\alpha}u := \partial_t I_t^\alpha u,
    \quad
    \bigl(I_t^\alpha u\bigr)(t) := \int_0^t\omega_\alpha(t-s)u(s)\,\mathrm ds,
    \quad 0<\alpha<1,
\end{equation}
with $\omega_\alpha(t):=t^{\alpha-1}/\Gamma(\alpha)$.

The resulting mass-conservative subdiffusive tumor growth model coupled to nutrient dynamics takes the form
\begin{subequations} \label{eq:subdiffusion-tumor}
\begin{align}
\partial_t \phi& = M_\phi {^R\partial_t^{1-\alpha}} \left( \Delta \mu_\phi \right)+   {^R\partial_t^{1-\alpha}} \left( P(\phi) \left(\mu_\sigma - \mu_\phi \right) \right), \label{eq:tumor-phi}\\
\mu_\phi & =-\epsilon_\phi^2 \Delta \phi+f(\phi) +\partial_\phi \chi  , \notag\\
\partial_t  \sigma& = M_\sigma{^R\partial_t^{1-\alpha}}\left( \Delta  \mu_\sigma \right)  -   {^R\partial_t^{1-\alpha}} \left( P(\phi) \left(\mu_\sigma - \mu_\phi \right) \right), \label{eq:tumor-sigma}\\
\mu_\sigma & =\frac{1}{\delta}\sigma + \partial_\sigma \chi .\notag
\end{align}
\end{subequations} 
Applying $I_t^{1-\alpha}$ to \eqref{eq:tumor-phi} and \eqref{eq:tumor-sigma} yields the equivalent Caputo formulation:
\begin{equation}\label{CapTumorSystem}
    {^C\partial_t^\alpha} \phi = M_\phi \Delta \mu_\phi + P(\phi) \left(\mu_\sigma - \mu_\phi \right), \quad
    {^C\partial_t^\alpha} \sigma = M_\sigma \Delta \mu_\sigma - P(\phi) \left(\mu_\sigma - \mu_\phi \right).
\end{equation}

\subsection{Asymptotically Compatible Variational Energy Dissipation Law}\label{subsec:original-energy-dissipation} 
We first verify mass conservation for \eqref{eq:subdiffusion-tumor}, then construct a variational energy and establish the corresponding dissipation law.

We impose either periodic boundary conditions or homogeneous Neumann conditions $\partial_{\bm n}\mu_\phi=\partial_{\bm n}\mu_\sigma=\partial_{\bm n}\phi=0$ on $\partial\Omega$. 
Integrating the sum of \eqref{eq:tumor-phi} and \eqref{eq:tumor-sigma} over $\Omega$ and using $S_\phi=-S_\sigma$ yields conservation of the total mass.

To construct the variational energy, we first derive an evolution identity for the free energy $E$.
Taking the $L^2$ inner products of \eqref{eq:tumor-phi} and \eqref{eq:tumor-sigma} with $\mu_\phi$ and $\mu_\sigma$, respectively, and integrating by parts yields
\begin{equation} \label{energy_dissipation}
    \begin{aligned}
        \frac{dE}{dt}
        =& - M_\phi \left( {^R\partial_t^{1-\alpha}}\nabla \mu_\phi,\nabla \mu_\phi \right)
        - M_\sigma \left( {^R\partial_t^{1-\alpha}} \nabla \mu_\sigma,\nabla \mu_\sigma \right) \\
        & - \left( {^R\partial_t^{1-\alpha}} \left( P(\phi) \left(\mu_\sigma - \mu_\phi \right) \right),\mu_\sigma - \mu_\phi \right).
    \end{aligned}
\end{equation}

\begin{lemma}\label{lem:weighted-RL-identity}
    Let $0<\alpha<1$, $\varrho,v\in C^1([0,T])$, and
    $\varrho(t)\ge0$ on $[0,T]$. Then, for $t\in(0,T]$, the following identity holds:
    \begin{align}
        v(t)\bigl({}^R\partial_t^{1-\alpha}(\varrho v)\bigr)(t)
        &={}
        \frac12\bigl({}^R\partial_t^{1-\alpha}(\varrho v^2)\bigr)(t)
        +\frac{d}{dt}\mathcal J_{\varrho,v}(t)
        +\mathcal R_{\varrho,v}(t),
        \label{eq:weighted-RL-identity-B}
    \end{align}
    with
    \begin{align*}
        \mathcal J_{\varrho,v}(t)&:=\frac12\int_0^t v^2(s)\bigl({}^R\partial_s^{1-\alpha}\varrho\bigr)(s)\,ds,\\
        \mathcal R_{\varrho,v}(t)
        &:={}
        \frac{1-\alpha}{2\Gamma(\alpha)}
        \int_0^t
        \frac{\varrho(s)(v(t)-v(s))^2}{(t-s)^{2-\alpha}}\,ds.
    \end{align*}
\end{lemma}

\begin{proof}
    We start from the identity
    \begin{equation}\label{eq:weighted-RL-pre}
        \begin{aligned}
            v(t)\bigl({}^R\partial_t^{1-\alpha}(\varrho\,v)\bigr)(t)
            &=
            \frac12\bigl({}^R\partial_t^{1-\alpha}(\varrho\,v^2)\bigr)(t)
            +\frac12\,\omega_\alpha(t)\varrho(t)v^2(t) \\
            &\quad
            +\frac{1-\alpha}{2\Gamma(\alpha)}
            \int_0^t
            \frac{
            \varrho(s)\bigl(v(t)-v(s)\bigr)^2+\bigl(\varrho(t)-\varrho(s)\bigr)v^2(t)
            }{(t-s)^{2-\alpha}}\,ds.
        \end{aligned}
    \end{equation} 
    Using
    \[
    {}^R\partial_t^{1-\alpha}\varrho(t)
    =  \omega_\alpha(t)\varrho(t)
    +\frac{1-\alpha}{\Gamma(\alpha)}
    \int_0^t
    \frac{\varrho(t)-\varrho(s)}{(t-s)^{2-\alpha}}\,ds,
    \]
    we obtain
    \begin{equation*}
    \frac12\,\omega_\alpha(t)\varrho(t)v^2(t)
    +\frac{1-\alpha}{2\Gamma(\alpha)}v^2(t)
    \int_0^t
    \frac{\varrho(t)-\varrho(s)}{(t-s)^{2-\alpha}}\,ds
    =
    \frac12\,v^2(t)\frac{d}{dt}(I_t^\alpha \varrho)(t).
    \end{equation*}
    Hence \eqref{eq:weighted-RL-pre} becomes
    \begin{equation}\label{eq:weighted-RL-pre2}
    \begin{aligned}
    v(t)\bigl({}^R\partial_t^{1-\alpha}(\varrho\,v)\bigr)(t)
    &=
    \frac12\bigl({}^R\partial_t^{1-\alpha}(\varrho\,v^2)\bigr)(t)
    +\frac12\,v^2(t)\frac{d}{dt}(I_t^\alpha \varrho)(t) \\
    &\quad
    +\frac{1-\alpha}{2\Gamma(\alpha)}
    \int_0^t
    \frac{\varrho(s)\bigl(v(t)-v(s)\bigr)^2}{(t-s)^{2-\alpha}}\,ds.
    \end{aligned}
    \end{equation}
    By the definition of $\mathcal J_{\varrho,v}$, for $t>0$,
    \begin{equation}\label{eq:J-derivative}
    \frac{d}{dt}\mathcal J_{\varrho,v}(t)=\frac12v^2(t)\bigl({}^R\partial_t^{1-\alpha}\varrho\bigr)(t)=\frac12v^2(t)\frac{d}{dt}(I_t^\alpha\varrho)(t).
    \end{equation}
    Substituting \eqref{eq:J-derivative} into \eqref{eq:weighted-RL-pre2} gives \eqref{eq:weighted-RL-identity-B}.
\end{proof}

Motivated by the weighted identity in Lemma~\ref{lem:weighted-RL-identity}, we define the continuous variational energy
\begin{equation}\notag
    {\mathcal{E}}_{\alpha}[\phi,\sigma](t)
    := E[\phi,\sigma](t)+\frac{1}{2}\left(I_t^\alpha \mathcal D \right)(t) + \int_\Omega \mathcal J_{P(\phi),\,\mu_\sigma-\mu_\phi}(\bm{x},t)\, d\bm{x},
\end{equation} 
where the total dissipation density $\mathcal D(t) = \mathcal D_{\mathrm{diff}}(t) + \mathcal D_P(t)$ is given by
\begin{equation}\label{eq:cont-dissipation-density}
    \mathcal D_{\mathrm{diff}}(t):=M_\phi\|\nabla\mu_\phi(t)\|^2+M_\sigma\|\nabla\mu_\sigma(t)\|^2,\quad
    \mathcal D_P(t):=\left\|\sqrt{P(\phi)}\left(\mu_\sigma-\mu_\phi\right)\right\|^2 .
\end{equation}

The Riemann--Liouville inequality~\cite{alsaedi2015maximum} states that
\begin{equation}\label{eq:RL-inequality}
    v(t)\left({^R\partial_t^{1-\alpha}}v\right)(t) \geq \frac{1}{2}\left({^R\partial_t^{1-\alpha}}v^2\right)(t) + \frac{1}{2}\omega_\alpha(t)v^2(t),\quad \forall v \in C^1([0,T]).
\end{equation}
Applying this inequality to the first two terms on the right-hand side of \eqref{energy_dissipation} and Lemma~\ref{lem:weighted-RL-identity} to the last term of \eqref{energy_dissipation} yields the variational energy dissipation law
\begin{equation}\notag
    \frac{d {\mathcal{E}}_\alpha}{dt}
    +\frac12\omega_\alpha(t)\mathcal D_{\mathrm{diff}}(t)
    +\int_\Omega\mathcal R_{P(\phi),\,\mu_\sigma-\mu_\phi}(\bm{x},t)\,d\bm{x}
    \le0.
\end{equation} 
Dropping the nonnegative reaction remainder yields
$\frac{d\mathcal E_\alpha}{dt}
+\frac12\omega_\alpha(t)\mathcal D_{\mathrm{diff}}(t)\le0$.

\begin{remark}[Integer-order limit]
    The model \eqref{eq:subdiffusion-tumor} and its variational energy dissipation law are asymptotically compatible with their integer-order counterparts as $\alpha\to1^{-}$.

    For fixed, sufficiently smooth functions $\phi$ and $\sigma$ independent of $\alpha$, and for $t>0$, as $\alpha\to1^-$,
    one has $\omega_\alpha(t)\to1$,
    $I_t^\alpha g\to\int_0^t g(s)\,ds$, and ${}^R\partial_t^{1-\alpha}g\to g$.
    Thus \eqref{eq:subdiffusion-tumor} recovers the corresponding integer-order model in~\cite{wang2025stability}.
    Moreover, for fixed $\varrho,v\in C^1([0,T])$ independent of $\alpha$, the representation ${}^R\partial_s^{1-\alpha}\varrho=\varrho(0)\omega_\alpha(s)+(I_s^\alpha\varrho')(s)$ and dominated convergence yield $\mathcal J_{\varrho,v}(t)\xrightarrow{\alpha\to1^-}\frac12\int_0^t\varrho(s)v^2(s)\,ds$.
    Since $(v(t)-v(s))^2\le\|v'\|_{\infty}^2(t-s)^2$, the integrand of $\mathcal R_{\varrho,v}$ is bounded by $\|\varrho\|_\infty\|v'\|_\infty^2(t-s)^{\alpha}$, and hence $\mathcal R_{\varrho,v}(t)\xrightarrow{\alpha\to1^-}0$.
    Consequently,
    \[
    \mathcal E_\alpha(t)\xrightarrow{\alpha\to1^-}
    E(t)+\frac12\int_0^t\mathcal D_{\mathrm{diff}}(s)\,ds
    +\int_0^t\mathcal D_P(s)\,ds.
    \]
    Thus, as $\alpha\to1^{-}$, the variational energy dissipation law reduces to the classical energy dissipation law:
    \[
    \frac{dE}{dt}+\mathcal D_{\mathrm{diff}}+\mathcal D_P\le0.
    \]
\end{remark}

\begin{remark}[Energy dissipation for a constant proliferation coefficient]
\leavevmode\par\nobreak\noindent
    For a constant proliferation coefficient $P(\phi)\equiv\bar P\ge0$, applying the Riemann--Liouville inequality \eqref{eq:RL-inequality} also to the reaction term yields
    \[
        \frac{d\mathcal E_{\alpha,\bar P}}{dt}+\frac12\omega_\alpha(t)\mathcal D(t)\le0,
        \quad
        \mathcal E_{\alpha,\bar P}(t):=E[\phi,\sigma](t)+\frac12\bigl(I_t^\alpha\mathcal D\bigr)(t).
    \]
    As $\alpha\to1^-$, both laws reduce to the integer-order law in~\cite{wang2025stability}.
\end{remark}

\section{Structure-Preserving Numerical Schemes} \label{sec:numericalscheme}
In this section, we develop a linear relaxation scheme for the subdiffusive tumor growth model \eqref{eq:subdiffusion-tumor}.
We then introduce temporal damping and augment the damped scheme with an a posteriori relaxation correction.
All three schemes conserve total mass and satisfy the discrete variational energy dissipation laws proved below.
Both the schemes and their dissipation laws are asymptotically compatible with their integer-order counterparts as $\alpha\to1^{-}$.
Throughout this section, we use the chemotactic coupling energy $\chi(\phi,\sigma)=-\chi_0\phi\sigma$~\cite{wang2025stability}, where $\chi_0\ge0$ is the chemotactic coupling coefficient.

\subsection{Relaxation Reformulation of the Continuous Model}
Following the linear relaxation idea in~\cite{yu2026linear}, we introduce a regularization parameter $\beta$ and an auxiliary variable $r(\phi)$ to rewrite the nonlinear term $f(\phi)$:
\begin{equation}\label{eq:relaxation-variable}
    r(\phi)=\frac12\phi(1-\phi)+\beta,\quad
    F(\phi)=(r-\beta)^2,\quad
    f(\phi)=(r-\beta)(1-2\phi).
\end{equation}
Using $r(\phi)$, we rewrite \eqref{eq:subdiffusion-tumor} by replacing
$\mu_\phi$ with
\[
    \tilde\mu_\phi
    :=-\epsilon_\phi^2\Delta\phi
    +(r-\beta)(1-2\phi)+\partial_\phi\chi.
\]
The corresponding relaxation energy is defined by
\begin{equation}\label{modienergy}
    \begin{aligned}
    \widetilde E(\phi,\sigma,r)
    := {}& \int_\Omega \Bigl[
        \frac{\epsilon_\phi^2}{2}|\nabla\phi|^2
        +2(r-\beta)\Bigl(\frac12\phi(1-\phi)+\beta\Bigr)-r^2 \\
    &\quad
        +\frac{\sigma^2}{2\delta}+\chi(\phi,\sigma)
        \Bigr]\,d\bm{x}+\beta^2|\Omega|.
    \end{aligned}
\end{equation}

Substituting $r=r(\phi)$ into \eqref{modienergy} gives $\widetilde E[\phi,\sigma,r(\phi)]=E[\phi,\sigma]$, and the reformulated chemical potential satisfies $\tilde\mu_\phi=\mu_\phi$.
Thus, the continuous reformulation of the original model preserves the original free energy, energy evolution identity, and mass conservation.

Applying the argument of Subsection~\ref{subsec:original-energy-dissipation}
with $E$ and $\mu_\phi$ replaced by $\widetilde E$ and $\tilde\mu_\phi$ gives the
same memory-augmented variational structure. Define $\widetilde{\mathcal D}$ and
$\widetilde{\mathcal D}_{\mathrm{diff}}$ by replacing $\mu_\phi$ with
$\tilde\mu_\phi$ in $\mathcal D$ and $\mathcal D_{\mathrm{diff}}$,
respectively. In terms of the auxiliary variable, the variational energy takes the form
\begin{equation*}
    \widetilde{\mathcal E}_{\alpha}[\phi,\sigma,r](t)
    :={}
    \widetilde E[\phi,\sigma,r](t)
    +\frac12\left(I_t^\alpha\widetilde{\mathcal D}\right)(t)
    +\int_\Omega
    \mathcal J_{P(\phi),\,\mu_\sigma-\tilde\mu_\phi}\,d\bm{x}.
\end{equation*}
The dissipation and reaction-history terms also coincide, so the continuous
relaxation reformulation preserves the variational energy exactly:
$\widetilde{\mathcal E}_\alpha[\phi,\sigma,r(\phi)]
=\mathcal E_\alpha[\phi,\sigma]$.
The relaxed system therefore satisfies the variational energy dissipation inequality
\begin{equation}\label{eq:relaxed-continuous-variational-energy} 
    \frac{d\widetilde{\mathcal E}_\alpha}{dt}
    +\frac12\omega_\alpha(t)\widetilde{\mathcal D}_{\mathrm{diff}}(t)
    +\int_\Omega\mathcal R_{P(\phi),\,\mu_\sigma-\tilde\mu_\phi}(\bm{x},t)\,d\bm{x}
    \le0.
\end{equation}

\subsection{A Discrete Weighted Riemann--Liouville Identity}
We discretize the Riemann--Liouville derivative ${}^R\partial_t^{1-\alpha}v=\partial_t I_t^\alpha v$ using the nonuniform $L1_R$ formula~\cite{liao2021energy}.
Let $0=t_0<t_1<\cdots<t_N=T$ be a nonuniform time grid with $\tau_n:=t_n-t_{n-1}$, $1\le n\le N$, and $\tau_*:=\max_{1\le k\le N}\tau_k$.
For $u^n\approx u(t_n)$, define $\nabla_\tau u^{n+1}:=u^{n+1}-u^n$, $\partial_\tau u^{n+\frac12}:=\nabla_\tau u^{n+1}/\tau_{n+1}$, and $u^{n+\frac12}:=(u^{n+1}+u^n)/2$.
Let $\Pi_0v$ denote the piecewise constant interpolant $(\Pi_0v)(t)=v^{k-\frac12}$ for $t\in(t_{k-1},t_k], 1\le k\le N$.
Approximating $v$ by $\Pi_0v$ in the fractional integral gives
\[
    \left(I_t^\alpha v\right)(t_n)
    \approx \left(I_\tau^\alpha v\right)^n
    :=\sum_{k=1}^n
    \varpi_{n-k}^{(n,\alpha)}v^{k-\frac12},
    \quad
    \varpi_{n-k}^{(n,\alpha)}
    :=\int_{t_{k-1}}^{t_k}\omega_\alpha(t_n-s)\,ds>0.
\]
Set $(I_\tau^\alpha v)^0:=0$.
The $L1_R$ formula is then defined by
\begin{equation}\label{eq:RL-direct-discrete}
    \left({}^R D_\tau^{1-\alpha}v\right)^{n-\frac12}
    :=\frac{\left(I_\tau^\alpha v\right)^n
    -\left(I_\tau^\alpha v\right)^{n-1}}{\tau_n}
    =\frac{1}{\tau_n}\sum_{k=1}^n
    a_{n-k}^{(n,\mathrm{RL})}v^{k-\frac12},
    \quad 1\le n\le N,
\end{equation}
where the convolution kernels satisfy
\begin{equation}\label{eq:RL-direct-kernels}
    a_0^{(n,\mathrm{RL})}:=\varpi_0^{(n,\alpha)}>0,
    \quad
    a_{n-k}^{(n,\mathrm{RL})}
    :=\varpi_{n-k}^{(n,\alpha)}
    -\varpi_{n-k-1}^{(n-1,\alpha)}<0,
    \quad 1\le k<n\le N.
\end{equation}
The following lemma gives a quantitative lower bound for the discrete quadratic form.
\begin{lemma}[\cite{liao2021energy}]\label{lem:RL-direct-kernel-positive}
    For $1\le n\le N$ and any real sequence $\{z_k\}_{k=1}^n$,
    the kernels in \eqref{eq:RL-direct-kernels} satisfy
    \begin{equation}\notag
            2z_k\sum_{j=1}^k a_{k-j}^{(k,\mathrm{RL})}z_j
            \ge \ z_k^2\sum_{j=1}^k a_{k-j}^{(k,\mathrm{RL})}
            + \sum_{j=1}^k \varpi_{k-j}^{(k,\alpha)}z_j^2  - \sum_{j=1}^{k-1} \varpi_{k-j-1}^{(k-1,\alpha)}z_j^2,
            \quad 1\le k\le n.
    \end{equation}
    Consequently,
    \begin{equation*}
        2\sum_{k=1}^n z_k
        \sum_{j=1}^k a_{k-j}^{(k,\mathrm{RL})}z_j
        \ge \sum_{k=1}^n
        \left(\varpi_{n-k}^{(n,\alpha)}
        +\sum_{j=1}^k a_{k-j}^{(k,\mathrm{RL})}\right)z_k^2.
    \end{equation*}
    The right-hand side is strictly positive whenever
    $\{z_k\}_{k=1}^n\not\equiv0$.
\end{lemma}

Taking $z_j=v^{j-\frac12}$ in Lemma~\ref{lem:RL-direct-kernel-positive}, we obtain the following positivity property of the $L1_R$ formula:
\begin{equation}\notag
    \sum_{j=1}^n \tau_j v^{j-\frac{1}{2}}\left({}^R D_\tau^{1-\alpha}v\right)^{j-\frac{1}{2}} = \int_0^{t_n} (\Pi_0 v)(t)\frac{\partial}{\partial t}\int_0^t \omega_\alpha(t-s)(\Pi_0 v)(s)\,ds\,dt \ge 0,
    \quad n\ge 1.
\end{equation}

To handle the time-dependent reaction coefficient, we establish a weighted discrete Riemann--Liouville identity.
\begin{lemma}[Discrete weighted Riemann--Liouville identity]\label{thm:discrete-weighted-RL}
    Let $\{\varrho^{n+\frac12}\}_{n=0}^{N-1}$ be a nonnegative sequence and $\{v^{n+\frac12}\}_{n=0}^{N-1}$ be a real sequence. 
    Then, for $0\le n\le N-1$,
    \begin{equation}\label{eq:discrete-weighted-RL-identity}
        v^{n+\frac12}\Bigl({}^R D_\tau^{1-\alpha}(\varrho v)\Bigr)^{n+\frac12}
        =
        \frac12\Bigl({}^R D_\tau^{1-\alpha}(\varrho v^2)\Bigr)^{n+\frac12}
        +\partial_\tau \mathcal J_{\tau}^{n+\frac12}[\varrho,v]
        +\mathcal R_{\tau}^{n+\frac12}[\varrho,v].
    \end{equation}
    Here
    \[
        \partial_\tau \mathcal J_{\tau}^{n+\frac12}[\varrho,v]
        :=
        \frac{\mathcal J_{\tau}^{n+1}[\varrho,v]-\mathcal J_{\tau}^{n}[\varrho,v]}{\tau_{n+1}}, 
    \]
    with, for $0\le m\le N$,
    \begin{equation}\notag
        \mathcal J_{\tau}^{m}[\varrho,v]:=\frac12\sum_{k=1}^{m}\tau_k(v^{k-\frac12})^2\bigl({}^R D_\tau^{1-\alpha}\varrho\bigr)^{k-\frac12}.
    \end{equation}
    The remainder term is given by
    \begin{equation*}
        \mathcal R_{\tau}^{n+\frac12}[\varrho,v]
        :=
        -\frac{1}{2\tau_{n+1}}
        \sum_{k=1}^{n+1}
        a_{n+1-k}^{(n+1,\mathrm{RL})}\varrho^{k-\frac12}
        \Bigl(v^{n+\frac12}-v^{k-\frac12}\Bigr)^2 \geq 0.
    \end{equation*}
\end{lemma}

\begin{proof}
    Fix $0\le n\le N-1$.
    By \eqref{eq:RL-direct-discrete}, we obtain
    \begin{equation}\label{eq:discrete-weighted-proof-1}
        \begin{aligned}
            &v^{n+\frac12}\Bigl({}^R D_\tau^{1-\alpha}(\varrho v)\Bigr)^{n+\frac12}=\frac{1}{\tau_{n+1}}\sum_{k=1}^{n+1}a_{n+1-k}^{(n+1,\mathrm{RL})}\varrho^{k-\frac12}v^{n+\frac12}v^{k-\frac12}\\
            &\quad=\frac{1}{2\tau_{n+1}}\sum_{k=1}^{n+1}a_{n+1-k}^{(n+1,\mathrm{RL})}\varrho^{k-\frac12}\Bigl[(v^{n+\frac12})^2+(v^{k-\frac12})^2-\bigl(v^{n+\frac12}-v^{k-\frac12}\bigr)^2\Bigr]\\
            &\quad=\frac12\Bigl({}^R D_\tau^{1-\alpha}(\varrho v^2)\Bigr)^{n+\frac12}+\frac{(v^{n+\frac12})^2}{2}\Bigl({}^R D_\tau^{1-\alpha}\varrho\Bigr)^{n+\frac12}+\mathcal R_\tau^{n+\frac12}[\varrho,v].
        \end{aligned}
    \end{equation}
    By the definition of $\mathcal J_\tau^m[\varrho,v]$,
    \[
        \partial_\tau\mathcal J_\tau^{n+\frac12}[\varrho,v]=\frac12(v^{n+\frac12})^2\Bigl({}^R D_\tau^{1-\alpha}\varrho\Bigr)^{n+\frac12}.
    \]
    Substituting this identity into \eqref{eq:discrete-weighted-proof-1} proves \eqref{eq:discrete-weighted-RL-identity}.
    
    Finally, $\mathcal R_\tau^{n+\frac12}[\varrho,v]\ge0$, as the $k=n+1$ term vanishes and the remaining terms are nonnegative by \eqref{eq:RL-direct-kernels} and $\varrho^{k-\frac12}\ge0$.
\end{proof}

\subsection{Linear Relaxation Scheme} 
For $1\le n\le N-1$, define the explicit predictor and the evaluated proliferation coefficient
\begin{equation}\label{eq:original-extrapolation}
    \widehat\phi^{n+\frac12}:=\phi^n+\frac{\tau_{n+1}}{2\tau_n}\left(\phi^n-\phi^{n-1}\right),
    \quad \widehat P^{n+\frac12}:=P\!\left(\widehat\phi^{n+\frac12}\right),
\end{equation}
and the mesh parameters
\begin{equation}\label{eq:weighted-r-mesh}
    \rho_{n+1}:=\frac{\tau_{n+1}}{\tau_n},
    \quad w_n:=\frac{1}{1+\rho_{n+1}},
    \quad c_n:=1-2w_n=\frac{\tau_{n+1}-\tau_n}{\tau_n+\tau_{n+1}}.
\end{equation}
We write $\chi^n:=\chi(\phi^n,\sigma^n)$ and denote by $\partial_\phi\chi^{n+\frac12}$ and $\partial_\sigma\chi^{n+\frac12}$ the partial derivatives of $\chi$ evaluated at $(\phi^{n+\frac12},\sigma^{n+\frac12})$.

The $L1_R$ discretization and the relaxation update on nonuniform meshes give
\begin{subequations} \label{CapTumorSystem_Relaxation_Discret}
    \begin{align}
        \partial_\tau\phi^{n+\frac12}
        &=\Bigl({}^R D_\tau^{1-\alpha}
        \bigl(M_\phi\Delta\tilde\mu_\phi
        +\widehat P(\mu_\sigma-\tilde\mu_\phi)\bigr)
        \Bigr)^{n+\frac12}, \label{TSD1}\\
        \tilde \mu_\phi^{n+\frac{1}{2}}& =  -\epsilon_\phi^2 \Delta \phi ^{n+\frac{1}{2}} + \left(r^{n+\frac{1}{2}} - \beta \right)(1-2\phi^{n+\frac{1}{2}}) + 
        \partial_\phi \chi ^{n+\frac{1}{2}}, \label{TSD2}\\
        w_n r^{n+\frac12}+(1-w_n)r^{n-\frac12}&=r(\phi^n),\label{TSD3}\\
        \partial_\tau\sigma^{n+\frac12}
        &=\Bigl({}^R D_\tau^{1-\alpha}
        \bigl(M_\sigma\Delta\mu_\sigma
        -\widehat P(\mu_\sigma-\tilde\mu_\phi)\bigr)
        \Bigr)^{n+\frac12}, \label{TSD4}\\
        \mu_\sigma^{n+\frac{1}{2}}& =  \frac{1}{\delta} \sigma ^{n+\frac{1}{2}}   + \partial_\sigma \chi ^{n+\frac{1}{2}}.\label{TSD5}
    \end{align}
\end{subequations}
The corresponding discrete relaxation energy is defined by
\begin{multline}\label{dismodienergy}
    \widetilde E[\phi^{n+1},\sigma^{n+1},r^{n+\frac12}]
    :=\int_\Omega\Bigl[\frac{\epsilon_\phi^2}{2}|\nabla\phi^{n+1}|^2
    +2(r^{n+\frac12}-\beta)r(\phi^{n+1})\\
    -(r^{n+\frac12})^2+\frac{(\sigma^{n+1})^2}{2\delta}
    +\chi^{n+1}\Bigr]\,d\bm{x}
    +\beta^2|\Omega|.
\end{multline}
Using $F(\phi)=(r(\phi)-\beta)^2$ and completing the square in \eqref{dismodienergy}, we obtain
\[
\widetilde E[\phi^{n+1},\sigma^{n+1},r^{n+\frac12}]-E(\phi^{n+1},\sigma^{n+1})=-\|r^{n+\frac12}-r(\phi^{n+1})\|^2.
\]

Define the discrete variational energy 
\begin{equation}\label{eq:original-variational-energy}
    \left(\widetilde{\mathcal E}_{\tau,\alpha}\right)^{n+1}
    :=\widetilde E
    \left[\phi^{n+1},\sigma^{n+1},r^{n+\frac12}\right]
    +\frac12\sum_{k=1}^{n+1}
    \varpi_{n+1-k}^{(n+1,\alpha)}
    \widetilde{\mathcal D}^{k-\frac12}
    +\mathcal J_{\tau,P}^{\mathrm{reac},\,n+1},
\end{equation}
where
$\widetilde{\mathcal D}^{n+\frac12}
:=\widetilde{\mathcal D}_{\mathrm{diff}}^{n+\frac12}
+\widetilde{\mathcal D}_{\mathrm{reac}}^{n+\frac12}$,
with
\begin{align}\label{diffusive_term}
    \widetilde{\mathcal D}_{\mathrm{diff}}^{n+\frac12}
    &:=M_\phi\left\|\nabla\tilde\mu_\phi^{n+\frac12}\right\|^2
    +M_\sigma\left\|\nabla\mu_\sigma^{n+\frac12}\right\|^2,\\
    \widetilde{\mathcal D}_{\mathrm{reac}}^{n+\frac12}
    &:=\left\|\sqrt{\widehat P^{n+\frac12}}
    \left(\mu_\sigma^{n+\frac12}
    -\tilde\mu_\phi^{n+\frac12}\right)\right\|^2, \label{eq:lr-reaction-dissipation}\\
    \mathcal J_{\tau,P}^{\mathrm{reac},\,m}
    &:=\int_\Omega
    \mathcal J_\tau^m
    \!\left[\widehat P,\mu_\sigma-\tilde\mu_\phi\right]
    \,d\bm x,\quad 0\le m\le N.\label{eq:lr-reaction-memory}
\end{align}
Set $r^{-\frac12}:=r(\phi^0)$.
For $n=0$,
a provisional solve of \eqref{TSD1}, \eqref{TSD2}, \eqref{TSD4}, and \eqref{TSD5} with $\widehat P^{\frac12}=P(\phi^0)$ and $r^{\frac12}=r(\phi^0)$ gives $\phi_*^1$.
We then set $\widehat\phi^{\frac12}:=(\phi^0+\phi_*^1)/2$,
$\widehat P^{\frac12}:=P(\widehat\phi^{\frac12})$,
and $r^{\frac12}:=r(\widehat\phi^{\frac12})$,
and compute $\phi^1$ and $\sigma^1$ from \eqref{CapTumorSystem_Relaxation_Discret} without \eqref{TSD3};
see Algorithm~\ref{algrtm}.
The initial variational energy is $(\widetilde{\mathcal E}_{\tau,\alpha})^0:=\widetilde E[\phi^0,\sigma^0,r^{-\frac12}]=E(\phi^0,\sigma^0)$.
 
The following energy dissipation result is established under the restriction of nondecreasing time steps. 
The same time-step restriction was also imposed in the analysis of finite difference methods for nonlinear fractional differential equations~\cite{li2016finite}.

\begin{theorem}[Mass conservation and variational energy dissipation]
\label{thm:1}
    Under periodic or homogeneous Neumann boundary conditions, the scheme \eqref{CapTumorSystem_Relaxation_Discret} satisfies
    \begin{equation}\label{eq:lr-mass-conservation}
    (\phi^n+\sigma^n,1)=(\phi^0+\sigma^0,1),\quad
    1\le n\le N.
    \end{equation}
    Assume that $\widehat P^{k-\frac12}\ge0$ for $1\le k\le N$.
    For $1\le n\le N-1$, if $\tau_{n+1}\ge\tau_n$, then the linear relaxation (LR) scheme \eqref{CapTumorSystem_Relaxation_Discret} satisfies
    \begin{equation}\label{fractional_energy_dissipation}
            \partial_\tau\left(\widetilde{\mathcal E}_{\tau,\alpha}\right)^{n+\frac12} +\frac{1}{2\tau_{n+1}}\left(\int_{t_n}^{t_{n+1}}\omega_\alpha(s)\,ds\right)\widetilde{\mathcal D}_{\mathrm{diff}}^{n+\frac12} 
            +\mathcal R_{\tau,P}^{\mathrm{reac},\,n+\frac12}\le 0,
    \end{equation}
    where
    \begin{equation}\label{eq:lr-reaction-remainder}
    \mathcal R_{\tau,P}^{\mathrm{reac},\,n+\frac12}:=\int_\Omega\mathcal R_\tau^{n+\frac12}[\widehat P,\mu_\sigma-\tilde\mu_\phi]\,d\bm x \ge 0.
    \end{equation}
\end{theorem}

\begin{proof}
    Integrating \eqref{TSD1} and \eqref{TSD4} over $\Omega$ and adding the results, the reaction terms cancel and the Laplacian terms vanish by the boundary conditions, which gives $(\phi^{k+1}+\sigma^{k+1},1)=(\phi^k+\sigma^k,1)$.
    Summing over $k=0,\ldots,n-1$ yields \eqref{eq:lr-mass-conservation}.

    Taking the $L^2$ inner products of \eqref{TSD1} and \eqref{TSD4} with $\tau_{n+1}\tilde\mu_\phi^{n+\frac12}$ and $\tau_{n+1}\mu_\sigma^{n+\frac12}$, respectively, and integrating by parts gives
    \begin{equation}\label{proof1}
    \begin{aligned}
    &(\tilde\mu_\phi^{n+\frac12},\nabla_\tau\phi^{n+1})+(\mu_\sigma^{n+\frac12},\nabla_\tau\sigma^{n+1})\\
    &=-M_\phi\left(\sum_{k=1}^{n+1}a_{n+1-k}^{(n+1,\mathrm{RL})}\nabla\tilde\mu_\phi^{k-\frac12},\nabla\tilde\mu_\phi^{n+\frac12}\right)-M_\sigma\left(\sum_{k=1}^{n+1}a_{n+1-k}^{(n+1,\mathrm{RL})}\nabla\mu_\sigma^{k-\frac12},\nabla\mu_\sigma^{n+\frac12}\right)\\
    &\quad-\left(\sum_{k=1}^{n+1}a_{n+1-k}^{(n+1,\mathrm{RL})}\widehat P^{k-\frac12}(\mu_\sigma^{k-\frac12}-\tilde\mu_\phi^{k-\frac12}),\mu_\sigma^{n+\frac12}-\tilde\mu_\phi^{n+\frac12}\right).
    \end{aligned}
    \end{equation}
    Applying Lemma~\ref{lem:RL-direct-kernel-positive} to the first two terms on the right-hand side of \eqref{proof1} and
    Lemma~\ref{thm:discrete-weighted-RL} to the third yields
    \begin{equation}\label{ineq_left}
    \begin{aligned}
    \mathrm{RHS}\le{}&-\frac{\tau_{n+1}}2\partial_\tau(I_\tau^\alpha\widetilde{\mathcal D})^{n+\frac12}-\frac12\left(\int_{t_n}^{t_{n+1}}\omega_\alpha(s)\,ds\right)\widetilde{\mathcal D}_{\mathrm{diff}}^{n+\frac12}\\
    &\quad-\tau_{n+1}\left(\partial_\tau\mathcal J_{\tau,P}^{\mathrm{reac},\,n+\frac12}+\mathcal R_{\tau,P}^{\mathrm{reac},\,n+\frac12}\right).
    \end{aligned}
    \end{equation}
    Combining \eqref{eq:relaxation-variable} and \eqref{TSD3} yields
    \begin{equation}\label{proof2}
    \begin{aligned}
    &((r^{n+\frac12}-\beta)(1-2\phi^{n+\frac12}),\nabla_\tau\phi^{n+1}) =2(r^{n+\frac12}-\beta,r(\phi^{n+1})-r(\phi^n))\\
    &=2((r^{n+\frac12}-\beta)r(\phi^{n+1})-(r^{n-\frac12}-\beta)r(\phi^n),1)\\
    &\quad-(\|r^{n+\frac12}\|^2-\|r^{n-\frac12}\|^2)+c_n\|r^{n+\frac12}-r^{n-\frac12}\|^2.
    \end{aligned}
    \end{equation}
    Here we used the identity $2(a-b)(w_na+(1-w_n)b)=a^2-b^2+(2w_n-1)(a-b)^2,\ \forall a,b\in\mathbb R$.    
    Using \eqref{proof2} and \eqref{TSD3}, the left-hand side of \eqref{proof1} becomes
    \begin{equation}\label{right_dis}
    \mathrm{LHS}=\widetilde E[\phi^{n+1},\sigma^{n+1},r^{n+\frac12}]-\widetilde E[\phi^n,\sigma^n,r^{n-\frac12}]+c_n\|r^{n+\frac12}-r^{n-\frac12}\|^2.
    \end{equation}
    Combining \eqref{proof1}, \eqref{ineq_left}, and \eqref{right_dis},
    dividing by $\tau_{n+1}$, and noting that $c_n\ge0$ since
    $\tau_{n+1}\ge\tau_n$, we obtain
    \eqref{fractional_energy_dissipation}
    by \eqref{eq:original-variational-energy}.
\end{proof}

\begin{remark}[Local consistency of the relaxation relation]\label{remark:weighted-consistency}
    For exact nodal values of a sufficiently smooth $\phi$,
    Taylor expansion about $t_n$ gives
    \[
    \begin{aligned}
    &r(\phi^n)-w_nr(\phi^{n+\frac12})-(1-w_n)r(\phi^{n-\frac12})\\
    &=-\frac{\tau_n\tau_{n+1}}8\left((1-2\phi(t_n))\phi''(t_n)-(\phi'(t_n))^2\right)+O(\tau_n\tau_{n+1}(\tau_n+\tau_{n+1})).
    \end{aligned}
    \]
\end{remark}

\begin{remark}[Integer-order compatibility]\label{remark:lr-integer-compatibility}
For a fixed time grid satisfying $\tau_{k+1}\ge\tau_k$ for $1\le k\le N-1$, and fixed discrete sequences, as $\alpha\to1^-$, $\varpi_{n-k}^{(n,\alpha)}\to\tau_k$, $a_0^{(n,\mathrm{RL})}\to\tau_n$, and $a_{n-k}^{(n,\mathrm{RL})}\to0$ for $k<n$.
Hence $({}^R D_\tau^{1-\alpha}v)^{n+\frac12}\to v^{n+\frac12}$, recovering the integer-order LR scheme. 
Moreover, $\mathcal J_{\tau,P}^{\mathrm{reac},\,m} \to\frac12\sum_{k=1}^m\tau_k\widetilde{\mathcal D}_{\mathrm{reac}}^{k-\frac12}$ and $\mathcal R_{\tau,P}^{\mathrm{reac},\,n+\frac12}\to0$.
Writing $\widetilde E^m:=\widetilde E[\phi^m,\sigma^m,r^{m-\frac12}]$ for $0\le m\le N$, we obtain
\[
 \left(\widetilde{\mathcal E}_{\tau,1}\right)^{m}
 =\widetilde E^m+\frac12\sum_{k=1}^m\tau_k
 \bigl(\widetilde{\mathcal D}^{k-\frac12}
 +\widetilde{\mathcal D}_{\mathrm{reac}}^{k-\frac12}\bigr).
\]
Since $\tau_{n+1}^{-1}\int_{t_n}^{t_{n+1}}\omega_\alpha(s)\,ds \to1$, inequality \eqref{fractional_energy_dissipation} reduces to
\[
 \partial_\tau\widetilde E^{n+\frac12}
 +\widetilde{\mathcal D}_{\mathrm{diff}}^{n+\frac12}
 +\widetilde{\mathcal D}_{\mathrm{reac}}^{n+\frac12}\le0,
 \quad 1\le n\le N-1.
\]
\end{remark}

\subsection{Damped Linear Relaxation Scheme}\label{subsec:energy-compatible-damped-relaxation}
The relaxation update admits an alternating auxiliary-variable error that is not controlled by energy monotonicity.
To examine this behavior, define\begin{equation}\label{eq:weighted-r-error}
e_r^{n+\frac12}:=r^{n+\frac12}-r(\phi^{n+\frac12})
=-\rho_{n+1}e_r^{n-\frac12}+\eta_r^n,
\end{equation}
where $\eta_r^n:=(1+\rho_{n+1})r(\phi^n)-r(\phi^{n+\frac12})-\rho_{n+1}r(\phi^{n-\frac12})$.
Its homogeneous component $(-1)^n(\tau_{n+1}/\tau_1)e_r^{\frac12}$ enters the tumor phase chemical potential through $(1-2\phi^{n+\frac12})e_r^{n+\frac12}$.
Neglect $\eta_r^n$ and let $q$ denote the induced perturbation in the combined phase diffusion and reaction term at the current half level.
By \eqref{eq:RL-direct-kernels}, the two most recent convolution terms contribute $\left(a_0^{(n+1,\mathrm{RL})}-\rho_{n+1}^{-1}a_1^{(n+1,\mathrm{RL})}\right)q$ to the phase increment, 
with the coefficient tending to $1+\rho_{n+1}^{-1}$ as $\alpha\to0^+$ and to $\tau_{n+1}$ as $\alpha\to1^-$.
Thus, at small fractional orders, the alternating perturbation enters the phase increment with a weight that is no longer reduced by the step size, and it is carried into subsequent steps through the convolution history.
Linearizing $\eta_r^n$ with $r'(\phi^n)$ and $r'(\phi^{n\pm\frac12})$ all replaced by $r'(\phi^{n+\frac12})$ yields
\[
\delta\eta_r^n=-\frac14\bigl(1-2\phi^{n+\frac12}\bigr)\tau_{n+1}\bigl(\partial_\tau\delta\phi^{n+\frac12}-\partial_\tau\delta\phi^{n-\frac12}\bigr).
\]
This motivates the following temporal damping of the phase chemical potential, with $\kappa_\phi\ge0$:
\begin{equation}\label{eq:damped-correction}
    \mathcal T_{\phi,d}^{\frac12}:=0,
    \quad \mathcal T_{\phi,d}^{n+\frac12}:=\kappa_\phi\frac{\tau_*^2}{\tau_{n+1}}\left(\partial_\tau\phi^{n+\frac12}-\partial_\tau\phi^{n-\frac12}\right),
    \quad 1\le n\le N-1.
\end{equation}

The corrected phase chemical potential is defined for $0\le n\le N-1$ by
\begin{equation}\label{CapTumorSystem_Relaxation_Discret_Damped}
    \mu_{\phi,d}^{n+\frac12}
    :=-\epsilon_\phi^2\Delta\phi^{n+\frac12}
    +(r^{n+\frac12}-\beta)(1-2\phi^{n+\frac12})
    +\partial_\phi\chi^{n+\frac12}
    +\mathcal T_{\phi,d}^{n+\frac12}.
\end{equation}
The damped linear relaxation (DLR) scheme consists of \eqref{TSD1} and \eqref{TSD4} with $\tilde\mu_\phi$ replaced by $\mu_{\phi,d}$ at all current and history levels,
together with \eqref{CapTumorSystem_Relaxation_Discret_Damped}, \eqref{TSD3}, and \eqref{TSD5}.

Define the damped relaxation energy by
\[
    \widetilde E_d^{n+1}
    :=\widetilde E
    \left[\phi^{n+1},\sigma^{n+1},r^{n+\frac12}\right]
    +\frac{\kappa_\phi\tau_*^2}{2}
    \left\|\partial_\tau\phi^{n+\frac12}\right\|^2.
\]
The corresponding discrete variational energy is
\begin{equation}\label{eq:damped-variational-energy}
    \left(\widetilde{\mathcal E}_{\tau,\alpha,d}\right)^{n+1}
    :=\widetilde E_d^{n+1}
    +\frac12\sum_{k=1}^{n+1}
    \varpi_{n+1-k}^{(n+1,\alpha)}
    \widetilde{\mathcal D}_d^{k-\frac12}
    +\mathcal J_{\tau,P,d}^{\mathrm{reac},\,n+1},
\end{equation}
where
$\widetilde{\mathcal D}_d^{n+\frac12}
:=\widetilde{\mathcal D}_{\mathrm{diff},d}^{n+\frac12}
+\widetilde{\mathcal D}_{\mathrm{reac},d}^{n+\frac12}$.
The quantities
$\widetilde{\mathcal D}_{\mathrm{diff},d}$,
$\widetilde{\mathcal D}_{\mathrm{reac},d}$,
$\mathcal J_{\tau,P,d}^{\mathrm{reac}}$, and
$\mathcal R_{\tau,P,d}^{\mathrm{reac}}$
are defined by replacing $\tilde\mu_\phi$ with $\mu_{\phi,d}$
in \eqref{diffusive_term}, \eqref{eq:lr-reaction-dissipation}, \eqref{eq:lr-reaction-memory}, and \eqref{eq:lr-reaction-remainder}, respectively.

\begin{theorem}[Mass conservation and variational energy dissipation]
\label{thm:damped-energy-dissipation}
    Under periodic or homogeneous Neumann boundary conditions, the DLR scheme satisfies
    \begin{equation}\label{eq:damped-mass-conservation}
        (\phi^n+\sigma^n,1)=(\phi^0+\sigma^0,1),
        \quad 1\le n\le N.
    \end{equation}
    Assume that $\widehat P^{k-\frac12}\ge0$ for $1\le k\le N$.
    For $1\le n\le N-1$, if $\tau_{n+1}\ge\tau_n$, then
    \begin{equation}\label{eq:damped-energy-dissipation}
        \begin{aligned}
        \partial_\tau \left(\widetilde{\mathcal E}_{\tau,\alpha,d}\right)^{n+\frac12}
        &+\frac{1}{2\tau_{n+1}}
        \left(\int_{t_n}^{t_{n+1}}\omega_\alpha(s)\,ds\right)
        \widetilde{\mathcal D}_{\mathrm{diff},d}^{n+\frac12}\\
        & 
        +\frac{\kappa_\phi\tau_*^2}{2\tau_{n+1}}
        \left\|\partial_\tau\phi^{n+\frac12}
        -\partial_\tau\phi^{n-\frac12}\right\|^2
        +\mathcal R_{\tau,P,d}^{\mathrm{reac},\,n+\frac12}
        \le0.
        \end{aligned}
    \end{equation}
\end{theorem}

\begin{proof}
    The proof of mass conservation \eqref{eq:damped-mass-conservation} is the same as in Theorem~\ref{thm:1}, so we omit it here.
    The testing procedure leading to \eqref{proof1} applies with $\tilde\mu_\phi$ replaced by $\mu_{\phi,d}$.
    Combining the identity
    \begin{equation*} 
    \begin{aligned}
    (\mathcal T_{\phi,d}^{n+\frac12},\nabla_\tau\phi^{n+1}) ={}& \frac{\kappa_\phi\tau_*^2}{2}\left(\|\partial_\tau\phi^{n+\frac12}\|^2-\|\partial_\tau\phi^{n-\frac12}\|^2\right)\\
    &\quad+\frac{\kappa_\phi\tau_*^2}{2}\left\|\partial_\tau\phi^{n+\frac12}-\partial_\tau\phi^{n-\frac12}\right\|^2,
    \end{aligned}
    \end{equation*}
    with the calculation in \eqref{right_dis}, we obtain
    \[
    \begin{aligned}
        &(\mu_{\phi,d}^{n+\frac12},\nabla_\tau\phi^{n+1})
        +(\mu_\sigma^{n+\frac12},\nabla_\tau\sigma^{n+1})\\
        &\quad=\widetilde E_d^{n+1}-\widetilde E_d^n+c_n\|r^{n+\frac12}-r^{n-\frac12}\|^2
        +\frac{\kappa_\phi\tau_*^2}{2}
        \left\|\partial_\tau\phi^{n+\frac12}
        -\partial_\tau\phi^{n-\frac12}\right\|^2.
    \end{aligned}
    \]
    Applying Lemma~\ref{lem:RL-direct-kernel-positive} to the two diffusion terms and Lemma~\ref{thm:discrete-weighted-RL} to the reaction term, and using \eqref{eq:damped-variational-energy} together with $\tau_{n+1}\ge\tau_n$, proves \eqref{eq:damped-energy-dissipation}.
\end{proof}

\begin{remark}[Damping consistency and integer-order compatibility]
\label{remark:damped-uniform-compatibility}
For a sufficiently smooth function $\phi$ and time grids satisfying $\tau_n/\tau_{n+1}\le C$ with $C$ independent of the grid, 
substituting $\phi(t_j)$ into \eqref{eq:damped-correction} gives $\mathcal T_{\phi,d}^{n+\frac12}=O(\tau_*^2)$ as $\tau_*\to0$.

As in Remark~\ref{remark:lr-integer-compatibility}, we hold the time grid and discrete sequences fixed and let $\alpha\to1^-$ to recover the integer-order DLR scheme.
The corresponding limiting variational energy is
\[
    \left(\widetilde{\mathcal E}_{\tau,1,d}\right)^{m}
    =\widetilde E_d^m
    +\frac12\sum_{k=1}^m\tau_k
    \bigl(\widetilde{\mathcal D}_d^{k-\frac12}
    +\widetilde{\mathcal D}_{\mathrm{reac},d}^{k-\frac12}\bigr),
    \quad 1\le m\le N.
\]
Taking its discrete time difference and passing to the limit
in \eqref{eq:damped-energy-dissipation} yields
\[
    \partial_\tau\widetilde E_d^{n+\frac12}
    +\widetilde{\mathcal D}_d^{n+\frac12}
    +\frac{\kappa_\phi\tau_*^2}{2\tau_{n+1}}
    \left\|\partial_\tau\phi^{n+\frac12}
    -\partial_\tau\phi^{n-\frac12}\right\|^2
    \le0,
    \quad 1\le n\le N-1.
\]
\end{remark}

\subsection{Damped Modified Linear Relaxation Scheme}
\label{subsec:energy-compatible-damped-modified-relaxation}

The temporal damping introduced in Subsection~\ref{subsec:energy-compatible-damped-relaxation} leaves the auxiliary-variable recurrence \eqref{eq:weighted-r-error} unchanged,
so the homogeneous component $(-1)^n(\tau_{n+1}/\tau_1)e_r^{\frac12}$ remains in the auxiliary-variable error.
To reduce the auxiliary-variable error, we apply an a posteriori correction to the predicted auxiliary variable $\widetilde r^{n+\frac12}$ toward $r(\phi^{n+\frac12})$, subject to the discrete energy dissipation condition.
The correction is performed after the current linear system has been solved and therefore preserves the linear structure of the scheme.

Building on the DLR scheme, the damped modified linear relaxation (DMLR) scheme adds an a posteriori correction to the auxiliary variable and proceeds in two steps for $1\le n\le N-1$.

\emph{Step 1: Linear prediction.}
Solve \eqref{TSD1}, \eqref{TSD4}, and \eqref{TSD5}
with $\tilde\mu_\phi$ replaced by $\mu_{\phi,dm}$ and
\begin{subequations}\label{CapTumorSystem_Relaxation_Discret_Modi}
\begin{align}
    &w_n\widetilde r^{n+\frac12}+(1-w_n)r^{n-\frac12}
    =r(\phi^n),
    \label{DMTSD3}\\
    &\mu_{\phi,dm}^{n+\frac12}
    =-\epsilon_\phi^2\Delta\phi^{n+\frac12}
    +(\widetilde r^{n+\frac12}-\beta)(1-2\phi^{n+\frac12})
    +\partial_\phi\chi^{n+\frac12}
    +\mathcal T_{\phi,d}^{n+\frac12}.
    \label{DMTSD2}
\end{align}
\end{subequations}
Using the corrected value $r^{n-\frac12}$ from the preceding step, the DLR update determines $\widetilde r^{n+\frac12}$ explicitly and then computes $\phi^{n+1}$ and $\sigma^{n+1}$ through a linear solve.
Previously computed history fluxes remain unchanged.

\emph{Step 2: Auxiliary-variable correction.}
After the linear solve, update
\begin{equation}\label{eq:combined-r-correction}
    r^{n+\frac12}
    =\theta^{n+\frac12}\widetilde r^{n+\frac12}
    +(1-\theta^{n+\frac12})R^{n+\frac12},
    \quad R^{n+\frac12}:=r(\phi^{n+\frac12}),
\end{equation}
where $\theta^{n+\frac12}\in[0,1]$.

To examine the effect of the correction, define
\[
    \widetilde e_r^{n+\frac12}:=\widetilde r^{n+\frac12}-R^{n+\frac12}, \quad
    e_r^{n+\frac12}:=r^{n+\frac12}-R^{n+\frac12}.
\]
The correction \eqref{eq:combined-r-correction} gives $e_r^{n+\frac12}=\theta^{n+\frac12}\widetilde e_r^{n+\frac12}$, 
which yields $e_r^{n+\frac12}=-\theta^{n+\frac12}\rho_{n+1}e_r^{n-\frac12}+\theta^{n+\frac12}\eta_r^n$ upon substitution from \eqref{DMTSD3}.
The associated homogeneous recurrence has amplification factor $-\theta^{n+\frac12}\rho_{n+1}$, 
so over several steps its homogeneous component is $(-1)^n(\tau_{n+1}/\tau_1)(\prod_{j=1}^n\theta^{j+\frac12})e_r^{\frac12}$.

To select the correction parameter while preserving discrete energy dissipation,
we first evaluate the energy change induced by the correction.
Replacing $\widetilde r^{n+\frac12}$ with $r^{n+\frac12}(\theta)$ according to \eqref{eq:combined-r-correction} changes the relaxation energy by
\[
\Delta\widetilde E^{n+1}(\theta)=\|\widetilde r^{n+\frac12}\|^2-\|r^{n+\frac12}(\theta)\|^2-2\bigl(\widetilde r^{n+\frac12}-r^{n+\frac12}(\theta),r(\phi^{n+1})\bigr).
\]
The quantities $\widetilde{\mathcal D}_{\mathrm{diff},dm}$, $\widetilde{\mathcal D}_{\mathrm{reac},dm}$, $\widetilde{\mathcal D}_{dm}$, $\mathcal J_{\tau,P,dm}^{\mathrm{reac}}$, and $\mathcal R_{\tau,P,dm}^{\mathrm{reac}}$ are defined as in the DLR scheme,
with $\mu_{\phi,d}$ replaced by $\mu_{\phi,dm}$.
These quantities remain unchanged during the correction because the chemical potentials and stored diffusion and reaction terms are not recomputed.
Define  
\begin{equation}\label{eq:combined-correction-budget}
\begin{aligned}
Q^{n+\frac12}&:=\frac{\kappa_\phi\tau_*^2}{2}\|\partial_\tau\phi^{n+\frac12}-\partial_\tau\phi^{n-\frac12}\|^2+\tau_{n+1}\mathcal R_{\tau,P,dm}^{\mathrm{reac},\,n+\frac12}\\
&\quad+\frac12\left(\int_{t_n}^{t_{n+1}}\omega_\alpha(s)\,ds\right)\widetilde{\mathcal D}_{\mathrm{diff},dm}^{n+\frac12}+c_n\|\widetilde r^{n+\frac12}-r^{n-\frac12}\|^2.
\end{aligned}
\end{equation}
Define the admissible set
\begin{equation}\label{eq:combined-admissible-set}
\mathcal V_{dm}^n:=\{\theta\in[0,1]:\Delta\widetilde E^{n+1}(\theta)\le\lambda Q^{n+\frac12}\}, \quad
\theta^{n+\frac12}:=\min\mathcal V_{dm}^n,
\end{equation}
where $\lambda\in[0,1)$.
Since $\tau_{n+1}\ge\tau_n$ gives $c_n\ge0$, we have $Q^{n+\frac12}\ge0$; together with $\Delta\widetilde E^{n+1}(1)=0$, this shows that $\theta=1$ is admissible, and continuity makes $\mathcal V_{dm}^n$ a nonempty closed subset of $[0,1]$.

To state the energy dissipation law under this correction, define the DMLR discrete variational energy by
\begin{equation}\label{eq:modified-variational-energy}
\left(\widetilde{\mathcal E}_{\tau,\alpha,dm}\right)^{n+1}:=\widetilde E_{dm}^{n+1}+\frac12\sum_{k=1}^{n+1}\varpi_{n+1-k}^{(n+1,\alpha)}\widetilde{\mathcal D}_{dm}^{k-\frac12}+\mathcal J_{\tau,P,dm}^{\mathrm{reac},\,n+1},
\end{equation}
where the relaxation energy is evaluated using the corrected auxiliary variable:
\[
\widetilde E_{dm}^{n+1}:=\widetilde E\left[\phi^{n+1},\sigma^{n+1},r^{n+\frac12}\right]+\frac{\kappa_\phi\tau_*^2}{2}\left\|\partial_\tau\phi^{n+\frac12}\right\|^2.
\]
The DMLR scheme satisfies the following conservation and dissipation properties.

\begin{theorem}[Mass conservation and variational energy dissipation]\label{thm:3}
\leavevmode\par\nobreak\noindent
Let $\lambda\in[0,1)$.
Under periodic or homogeneous Neumann boundary conditions, the DMLR scheme \eqref{CapTumorSystem_Relaxation_Discret_Modi} with the correction rule \eqref{eq:combined-admissible-set} satisfies
\begin{equation}\label{eq:dmlr-mass-conservation}
(\phi^n+\sigma^n,1)=(\phi^0+\sigma^0,1),
\quad 1\le n\le N.
\end{equation}
Assume that $\widehat P^{k-\frac12}\ge0$ for $1\le k\le N$.
For $1\le n\le N-1$, if $\tau_{n+1}\ge\tau_n$, then
\begin{equation}\label{Modified_fractional_energy_dissipation_relaxation}
\begin{aligned}
&\partial_\tau\left(\widetilde{\mathcal E}_{\tau,\alpha,dm}\right)^{n+\frac12}+\frac{1-\lambda}{2\tau_{n+1}}\left(\int_{t_n}^{t_{n+1}}\omega_\alpha(s)\,ds\right)\widetilde{\mathcal D}_{\mathrm{diff},dm}^{n+\frac12}\\
&\quad+\frac{(1-\lambda)\kappa_\phi\tau_*^2}{2\tau_{n+1}}\left\|\partial_\tau\phi^{n+\frac12}-\partial_\tau\phi^{n-\frac12}\right\|^2+(1-\lambda)\mathcal R_{\tau,P,dm}^{\mathrm{reac},\,n+\frac12}\le0.
\end{aligned}
\end{equation}
\end{theorem}

\begin{proof}
    Since the correction leaves $\phi^{n+1}$ and $\sigma^{n+1}$ unchanged, the DMLR scheme preserves total mass by the same argument as for the DLR scheme.

    For the energy estimate, repeat the derivation of \eqref{eq:damped-energy-dissipation} for the predictor,
    using $\widetilde r^{n+\frac12}$ at the new half time level and the corrected value $r^{n-\frac12}$ at the preceding one.
    The correction changes the relaxation energy by $\Delta\widetilde E^{n+1}(\theta^{n+\frac12})$,
    while leaving the damping contribution and all dissipation and history terms unchanged.
    Retaining the term $c_n\|\widetilde r^{n+\frac12}-r^{n-\frac12}\|^2$ in the predictor estimate and using \eqref{eq:combined-correction-budget} and \eqref{eq:modified-variational-energy} yields
    \[
    \left(\widetilde{\mathcal E}_{\tau,\alpha,dm}\right)^{n+1}-\left(\widetilde{\mathcal E}_{\tau,\alpha,dm}\right)^n+Q^{n+\frac12}
    \le\Delta\widetilde E^{n+1}(\theta^{n+\frac12})
    \le\lambda Q^{n+\frac12},
    \]
    where the last inequality follows from \eqref{eq:combined-admissible-set}.
    Rearranging, substituting \eqref{eq:combined-correction-budget}, dropping $(1-\lambda)c_n\|\widetilde r^{n+\frac12}-r^{n-\frac12}\|^2\ge0$, and dividing by $\tau_{n+1}$ proves \eqref{Modified_fractional_energy_dissipation_relaxation}.
\end{proof}

\begin{remark}[Selection of the correction parameter] \label{remark:theta}
    Let $d^{n+\frac12}:=\widetilde r^{n+\frac12}-R^{n+\frac12}$.
    For $\theta\in[0,1]$, the energy constraint in \eqref{eq:combined-admissible-set} is equivalent to
    \[
    p_n(\theta):=A_n\theta^2+B_n\theta+C_n\le0,
    \]
    where
    \begin{equation}\label{eq:combined-theta-coefficients}
        \begin{aligned}
            A_n&=-\|d^{n+\frac12}\|^2,\quad
            B_n =2(d^{n+\frac12},r(\phi^{n+1})-R^{n+\frac12}),\\
            C_n&=\|\widetilde r^{n+\frac12}\|^2-\|R^{n+\frac12}\|^2-2(d^{n+\frac12},r(\phi^{n+1}))-\lambda Q^{n+\frac12}.
        \end{aligned}
    \end{equation}
    If $A_n=0$, then $d^{n+\frac12}=0$ and $p_n(\theta)=-\lambda Q^{n+\frac12}\le0$,
    so the minimum admissible value is $\theta=0$.
    If $A_n<0$ and $C_n\le0$,
    then $p_n(0)=C_n\le0$,
    so the minimum admissible value is again $\theta=0$.
    Otherwise, $A_n<0$ and $C_n>0$,
    so the two roots have opposite signs.
    Since $p_n(1)=-\lambda Q^{n+\frac12}\le0$,
    the positive root lies in $(0,1]$ and is the minimum admissible value.
    Therefore,
    \begin{equation}\label{choice_theta}
        \theta^{n+\frac12}=
        \begin{cases}
            0, & A_n=0\ \text{or}\ C_n\le0,\\[2mm]
            \displaystyle\frac{-B_n-\sqrt{B_n^2-4A_nC_n}}{2A_n}, & A_n<0\ \text{and}\ C_n>0.
        \end{cases}
    \end{equation}
\end{remark} 

\begin{remark}[Integer-order compatibility]\label{rem:compatibility_DMLR}
As stated in Remark~\ref{remark:lr-integer-compatibility}, the discrete history operators recover their integer-order limits as $\alpha\to1^-$ on a fixed time grid with fixed discrete sequences.
These limits yield the integer-order DMLR predictor.
The auxiliary correction retains its form, with $Q^{n+\frac12}$ in \eqref{eq:combined-admissible-set} replaced by its limit as $\alpha\to1^-$:
\[ 
Q_1^{n+\frac12} :=\frac{\kappa_\phi\tau_*^2}{2}\|\partial_\tau\phi^{n+\frac12}-\partial_\tau\phi^{n-\frac12}\|^2 +\frac{\tau_{n+1}}2\widetilde{\mathcal D}_{\mathrm{diff},dm}^{n+\frac12}+c_n\|\widetilde r^{n+\frac12}-r^{n-\frac12}\|^2.
\]
The limiting variational energy is
\[
 \left(\widetilde{\mathcal E}_{\tau,1,dm}\right)^{m}
 =\widetilde E_{dm}^m+\frac12\sum_{k=1}^m\tau_k
 (\widetilde{\mathcal D}_{dm}^{k-\frac12}
 +\widetilde{\mathcal D}_{\mathrm{reac},dm}^{k-\frac12}).
\]
Hence \eqref{Modified_fractional_energy_dissipation_relaxation} reduces to 
 \begin{equation*}
\begin{aligned}
\frac{\widetilde E_{dm}^{n+1}-\widetilde E_{dm}^n}{\tau_{n+1}}&+\left(1-\frac\lambda2\right)\widetilde{\mathcal D}_{\mathrm{diff},dm}^{n+\frac12}+\widetilde{\mathcal D}_{\mathrm{reac},dm}^{n+\frac12}\\
&\quad+\frac{(1-\lambda)\kappa_\phi\tau_*^2}{2\tau_{n+1}}\|\partial_\tau\phi^{n+\frac12}-\partial_\tau\phi^{n-\frac12}\|^2\le0,\quad 1\le n\le N-1.
\end{aligned}
 \end{equation*} 
Thus, on a fixed time grid with fixed discrete sequences, the DMLR predictor, correction budget, and energy inequality recover their integer-order counterparts.
\end{remark}

\section{Numerical Experiments}\label{sec:example}
In this section, we first present the temporal convergence of the three schemes and then examine variational energy dissipation and mass conservation in long-term tumor growth simulations.
We also compare model simulations with experimental data on in vitro tumor spheroid growth and investigate tumor growth under sustained nutrient supply.
The parameter $\beta=2$ is used throughout. Unless otherwise specified, the DMLR scheme is used for the numerical simulations.
The three schemes are summarized in Algorithm~\ref{algrtm}.

\begin{algorithm}
\caption{Unified time stepping for the LR, DLR, and DMLR schemes}
\label{algrtm}
\begin{algorithmic}[1]
\State \textbf{Input:}
Set scheme $\mathsf S\in\{\mathrm{LR},\mathrm{DLR},\mathrm{DMLR}\}$; time grid and model parameters; fixed damping scale $\tau_*$; initial values $\phi^0,\sigma^0$.
\State \textbf{Output:}
$\{\phi^n,\sigma^n\}_{n=1}^{N}$ and $\{r^{n+\frac12}\}_{n=0}^{N-1}$.
\For{$n=0,\ldots,N-1$}
    \If{$n=0$}
        \State Compute $\phi_*^1$ from \eqref{TSD1}, \eqref{TSD4}, and \eqref{TSD5} with $\widehat P^{\frac12}=P(\phi^0)$ and $r^{\frac12}=r(\phi^0)$.
        \State Set $\widehat\phi^{\frac12}=(\phi^0+\phi_*^1)/2$, $\widehat P^{\frac12}=P(\widehat\phi^{\frac12})$;
        use $r(\widehat\phi^{\frac12})$ for $r^{\frac12}$ (LR/DLR) or $\widetilde r^{\frac12}$ (DMLR).
    \Else
        \State Compute $\widehat\phi^{n+\frac12}$ and $\widehat P^{n+\frac12}$ by \eqref{eq:original-extrapolation}.
        \If{$\mathsf S=\mathrm{DMLR}$}
            \State Compute $\widetilde r^{n+\frac12}$ by \eqref{DMTSD3}.
        \Else
            \State Compute $r^{n+\frac12}$ by \eqref{TSD3}.
        \EndIf
    \EndIf
    \State Solve \eqref{TSD1}, \eqref{TSD4}, and \eqref{TSD5} linearly for $\phi^{n+1}$ and $\sigma^{n+1}$,
    with the tumor phase chemical potential defined by \eqref{TSD2} for LR, \eqref{CapTumorSystem_Relaxation_Discret_Damped} for DLR, and \eqref{DMTSD2} for DMLR.
    \If{$\mathsf S=\mathrm{DMLR}$}
        \State Set $R^{n+\frac12}=r(\phi^{n+\frac12})$, and compute $Q^{n+\frac12}$ by \eqref{eq:combined-correction-budget} for $n\ge1$ or set $Q^{\frac12}=0$.
        \State Choose $\theta^{n+\frac12}$ by \eqref{choice_theta} and update $r^{n+\frac12}$ by \eqref{eq:combined-r-correction}.
    \EndIf
\EndFor
\end{algorithmic}
\end{algorithm}

\begin{example}[Error and convergence comparison]\label{example1}
Suitable source terms are added to \eqref{eq:subdiffusion-tumor} to admit the exact solution
\begin{equation*}
\phi(x,y,t)=\omega_{1+v}(t)\frac{\sin(2\pi x)\sin(2\pi y)}4,\quad \sigma(x,y,t)=1-\phi(x,y,t),
\end{equation*}
on the unit periodic square with $T=1$ and $v=1$.
Parameters are chosen as $\epsilon_\phi=0.05$, $\delta=0.3$, $P_0=0.05$, $\chi_0=0.02$, $M_\phi=10^{-3}$, and $M_\sigma=10^{-1}$,
with the constant proliferation coefficient $P(\phi)\equiv P_0$.
We use the time grid $t_n=(n/N_t)^\gamma$,
with $N_t=8,16,32,64,128$,
on a $256\times256$ spatial grid.
Both damped schemes use $\kappa_\phi=4$ throughout this example.
For DMLR, the budget fraction is $\lambda=0.01$ in Figure~\ref{fig:example1-grading} and $\lambda=0.99$ in the remaining computations of this example.
Define $E_\phi=\max_{1\le n\le N_t}\max_{i,j}|\phi_{ij}^{n}-\phi(x_i,y_j,t_n)|$, and define $E_\sigma$ and $E_r$ analogously, using half time levels for $r$.
The observed convergence order is computed as $\log_2(E(N_t)/E(2N_t))$.
\end{example}

\begin{figure}[!htbp]
\centering
\includegraphics[width=.96\textwidth]{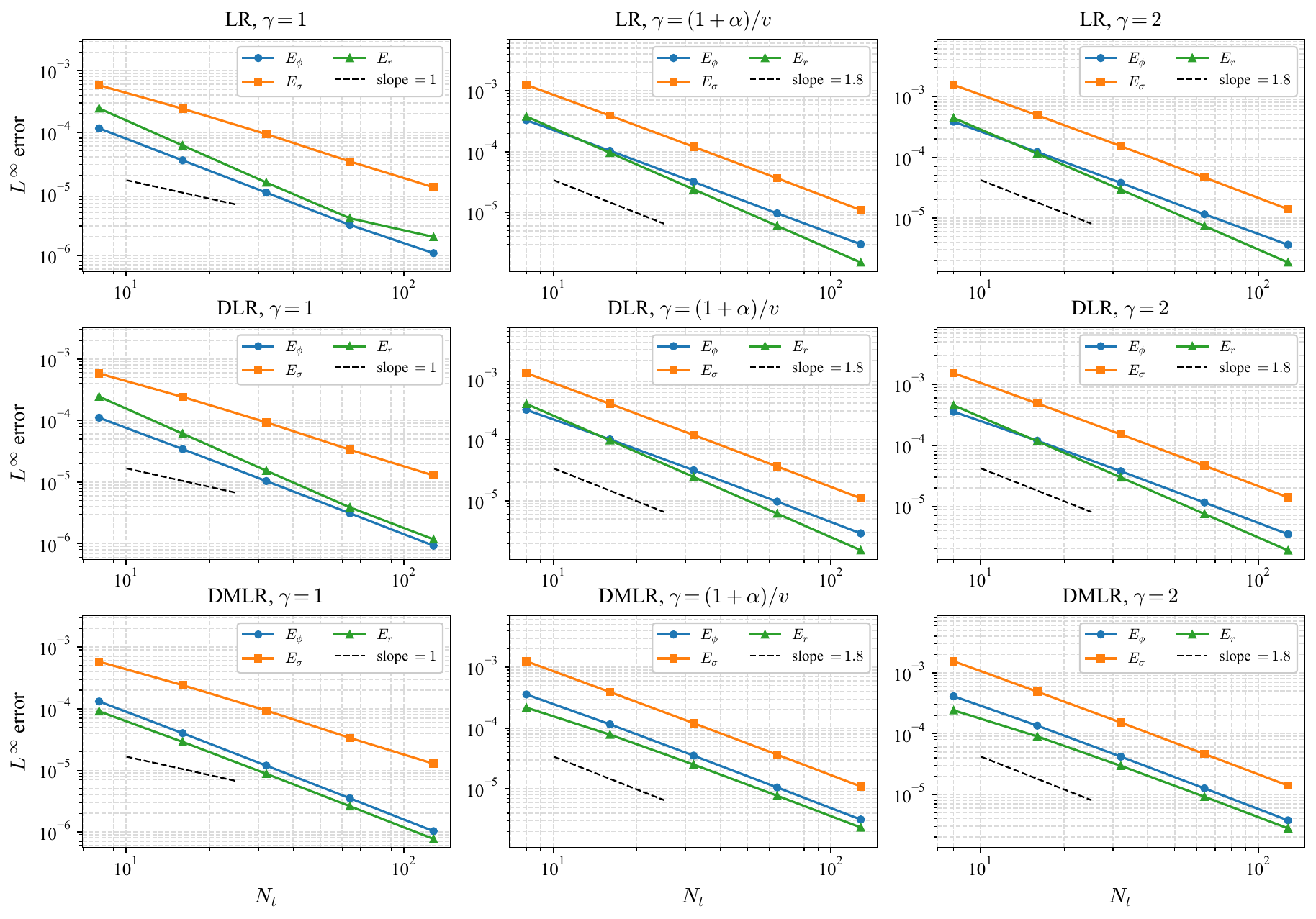}
\caption{Temporal $L^\infty$ errors and convergence rates of LR, DLR, and DMLR for $(\alpha,v)=(0.8,1)$.}\label{fig:example1-grading}
\end{figure}

Figure~\ref{fig:example1-grading} compares the temporal convergence of the LR, DLR, and DMLR schemes at $\alpha=0.8$.
The observed convergence rates for $\phi$, $\sigma$, and $r$ are consistent with the expected order $\min\{1+\alpha,\gamma v\}$.
All three schemes yield comparable errors in $\phi$ and $\sigma$.
DMLR produces smaller auxiliary-variable errors than LR and DLR for $\gamma=1$, whereas this advantage is largely restricted to coarse grids for $\gamma=(1+\alpha)/v$ and $\gamma=2$.

\begin{table}[!htbp]
\centering
\caption{Temporal $L^\infty$ errors and convergence rates of the LR scheme for $\alpha=0.1,0.2,0.3$ and $\gamma=2$.}\label{tab:weighted-e1-small-lr}
\scriptsize
\setlength{\tabcolsep}{3.2pt}
\renewcommand{\arraystretch}{1.08}
\begin{tabular}{cc rr rr rr l}
\toprule
$\alpha$ & $N_t$ & \multicolumn{2}{c}{$\phi$} & \multicolumn{2}{c}{$\sigma$} & \multicolumn{2}{c}{$r$} & Status\\
\cmidrule(lr){3-4}\cmidrule(lr){5-6}\cmidrule(lr){7-8}
&& Error & Order & Error & Order & Error & Order &\\
\midrule
\multirow{5}{*}{0.1}
& 8 & 5.01e-03 & -- & 2.27e-02 & -- & 2.52e-03 & -- & finite\\
& 16 & 3.06e-03 & 0.71 & 1.06e-02 & 1.09 & 2.70e-03 & -0.10 & finite\\
& 32 & 4.93e-02 & -4.01 & 4.94e-03 & 1.11 & 7.55e-02 & -4.81 & unstable\\
& 64 & \texttt{NaN} & -- & \texttt{NaN} & -- & \texttt{NaN} & -- & diverged\\
& 128 & \texttt{NaN} & -- & \texttt{NaN} & -- & \texttt{NaN} & -- & diverged\\
\midrule
\multirow{5}{*}{0.2}
& 8 & 3.71e-03 & -- & 1.71e-02 & -- & 1.70e-03 & -- & finite\\
& 16 & 1.83e-03 & 1.02 & 7.47e-03 & 1.19 & 1.42e-03 & 0.26 & finite\\
& 32 & 3.51e-03 & -0.94 & 3.30e-03 & 1.18 & 5.30e-03 & -1.90 & unstable\\
& 64 & \texttt{NaN} & -- & \texttt{NaN} & -- & \texttt{NaN} & -- & diverged\\
& 128 & \texttt{NaN} & -- & \texttt{NaN} & -- & \texttt{NaN} & -- & diverged\\
\midrule
\multirow{5}{*}{0.3}
& 8 & 2.73e-03 & -- & 1.26e-02 & -- & 1.11e-03 & -- & finite\\
& 16 & 1.20e-03 & 1.19 & 5.17e-03 & 1.29 & 8.01e-04 & 0.46 & finite\\
& 32 & 6.78e-04 & 0.82 & 2.14e-03 & 1.27 & 7.28e-04 & 0.14 & finite\\
& 64 & 6.98e-02 & -6.69 & 1.05e-03 & 1.03 & 1.33e-01 & -7.51 & unstable\\
& 128 & \texttt{NaN} & -- & \texttt{NaN} & -- & \texttt{NaN} & -- & diverged\\
\bottomrule
\end{tabular}
\par\smallskip
\begin{minipage}{\linewidth}
\footnotesize
\textit{Note.} \texttt{NaN} indicates solver nonconvergence or nonfinite values.
\end{minipage}
\end{table}

\begin{figure}[!htbp]
\centering
\includegraphics[width=.96\textwidth]{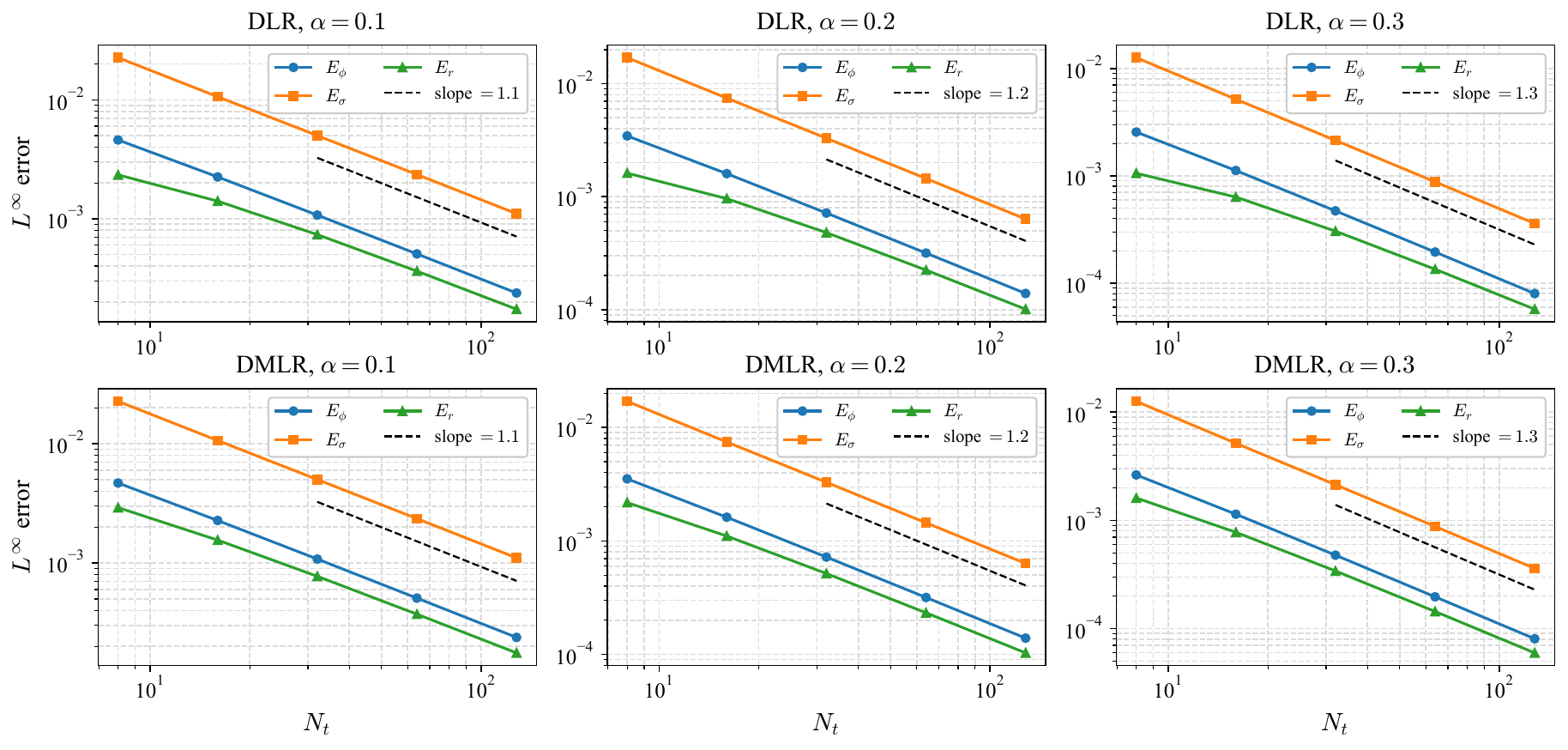}
\caption{Temporal $L^\infty$ errors and convergence rates of the DLR and DMLR schemes for $\alpha=0.1,0.2,0.3$ and $\gamma=2$.}\label{fig:example1-small}
\end{figure}

\begin{figure}[!htbp]
\centering
\begin{subfigure}[t]{0.32\textwidth}
\centering
\includegraphics[width=\linewidth]{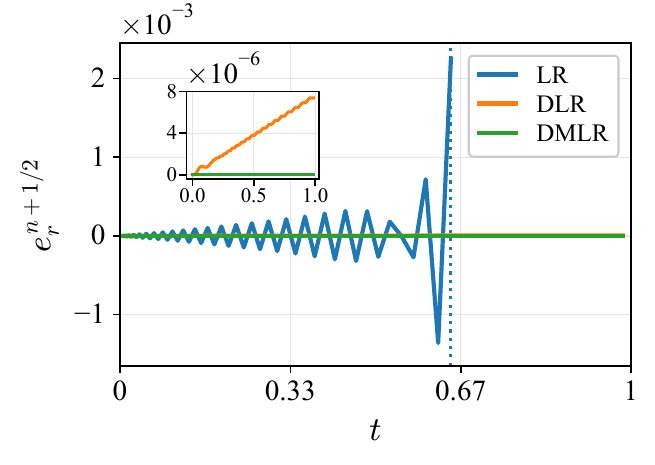}
\caption{}\label{fig:example1-mechanism-auxiliary}
\end{subfigure}
\hfill
\begin{subfigure}[t]{0.32\textwidth}
\centering
\includegraphics[width=\linewidth]{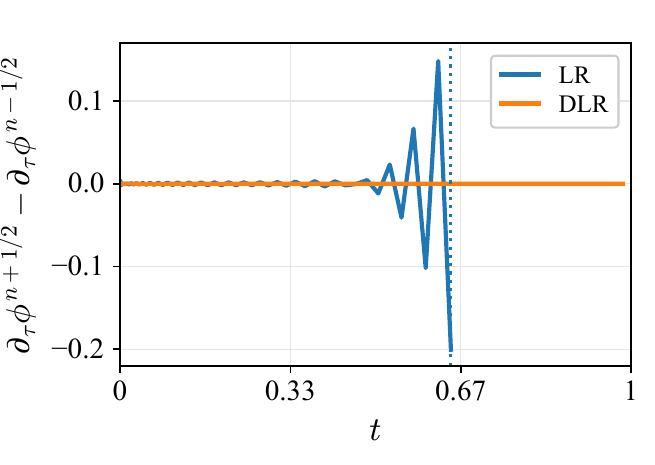}
\caption{}\label{fig:example1-mechanism-response}
\end{subfigure}
\hfill
\begin{subfigure}[t]{0.32\textwidth}
\centering
\includegraphics[width=\linewidth]{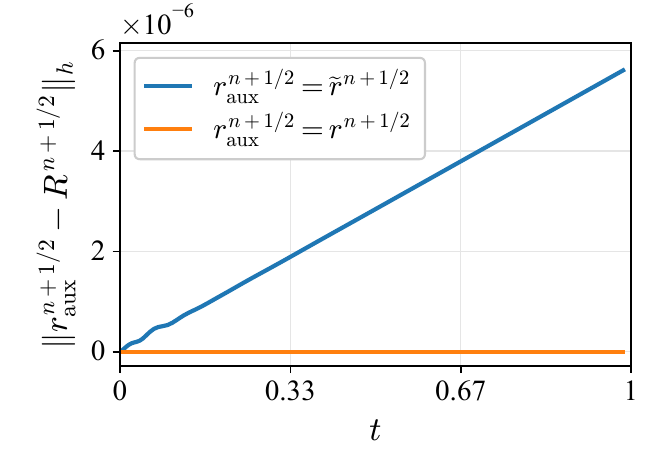}
\caption{}\label{fig:example1-mechanism-correction}
\end{subfigure}
\caption{Temporal damping and auxiliary-variable correction for $\alpha=0.1$, $\gamma=2$, and $N_t=64$:
(a) auxiliary-variable deviation;
(b) differences of consecutive discrete phase-field time derivatives;
(c) DMLR deviations from the same target $R^{n+\frac12}$ before ($r_{\mathrm{aux}}^{n+\frac12}=\widetilde r^{n+\frac12}$) and after ($r_{\mathrm{aux}}^{n+\frac12}=r^{n+\frac12}$) correction.}
\label{fig:example1-mechanisms}
\end{figure}

Table~\ref{tab:weighted-e1-small-lr} and Figure~\ref{fig:example1-small} compare the $L^\infty$ errors and convergence rates of these three schemes for small $\alpha$ values $0.1,0.2,0.3$ and $\gamma=2$.
The LR scheme exhibits error growth in $\phi$ and $r$ and eventual computational breakdown as the number of time steps increases.
DLR recovers the expected convergence rates for $\phi$ and $\sigma$, and the finest-grid orders for $r$ are $1.06$, $1.15$, and $1.24$, respectively.
With the additional relaxation correction, DMLR achieves convergence rates consistent with the predicted orders $1+\alpha$ for all three variables.
Figure~\ref{fig:example1-mechanisms} illustrates the auxiliary deviations and phase response for $\alpha=0.1$ and $N_t=64$.
Pointwise quantities are evaluated at $(1/4,1/4)$, with the auxiliary deviation defined by \eqref{eq:weighted-r-error}.
LR fails at step 53; dotted vertical lines indicate its last accepted half time level.
Subfigures~\ref{fig:example1-mechanisms}(\subref{fig:example1-mechanism-auxiliary},\subref{fig:example1-mechanism-response}) show alternating auxiliary deviations and oscillatory time-derivative differences in LR.
DLR reduces the maximum absolute time-derivative difference over $1\le n\le51$ from $2.01\times10^{-1}$ to $2.74\times10^{-4}$, consistent with temporal damping limiting the propagation of auxiliary deviations to the phase variable.
Subfigure~\ref{fig:example1-mechanisms}(\subref{fig:example1-mechanism-correction}) shows that DMLR selects $\theta^{n+\frac12}=0$ at every post-startup step, eliminating the auxiliary deviation relative to the numerical midpoint target.

\begin{example}[Tumor growth under a heterogeneous nutrient distribution]\label{example2}
\leavevmode\par\nobreak\noindent
We consider the tumor--nutrient model on $ [-\pi/2,\pi/2]^2$ with periodic boundary conditions and the following initial data:
\begin{equation*}
    \begin{aligned}
        &\phi_0(x,y)={}\frac12\Bigg[3+\tanh\!\left(\frac{0.12-\sqrt{{\left(\frac{x+0.3}{1.7}\right)}^{2}+{\left(\frac{y-0.3}{1.2}\right)}^{2}}}{\sqrt2\,\epsilon_\phi}\right)\\
        &\quad+\tanh\!\left(\frac{0.15-\sqrt{{\left(\frac{x-0.3}{1.5}\right)}^{2}+{\left(\frac{y-0.25}{1.5}\right)}^{2}}}{\sqrt2\,\epsilon_\phi}\right)+\tanh\!\left(\frac{0.1-\sqrt{{\left(\frac{x}{2}\right)}^{2}+{\left(\frac{y+0.4}{2.5}\right)}^{2}}}{\sqrt2\,\epsilon_\phi}\right)\Bigg],\\
        &\sigma_0(x,y)={}1-0.9\phi_0(x,y).
    \end{aligned}
\end{equation*}
In the implementation, the prescribed $\phi_0$ is first clipped to $[0,1]$, and values below $0.05$ are set to zero before evaluating $\sigma_0$.
Parameters are chosen as $\epsilon_\phi=0.02$, $\delta=0.5$, $P_0=0.8$, $\chi_0=0.01$, $M_\phi=M_\sigma=10^{-3}$,  $\kappa_\phi=1$, and $\lambda=0.95$ on a $256\times256$ Fourier grid.
We simulate up to $T=30$ using 340 nondecreasing time steps,
with $\tau_1=0.06$ and $\tau_*=0.09375$.
The proliferation rate is given by $P(\phi)=P_0\delta\phi^2$.
\end{example}

\begin{figure}[!htbp]
\centering
\includegraphics[width=.94\textwidth]{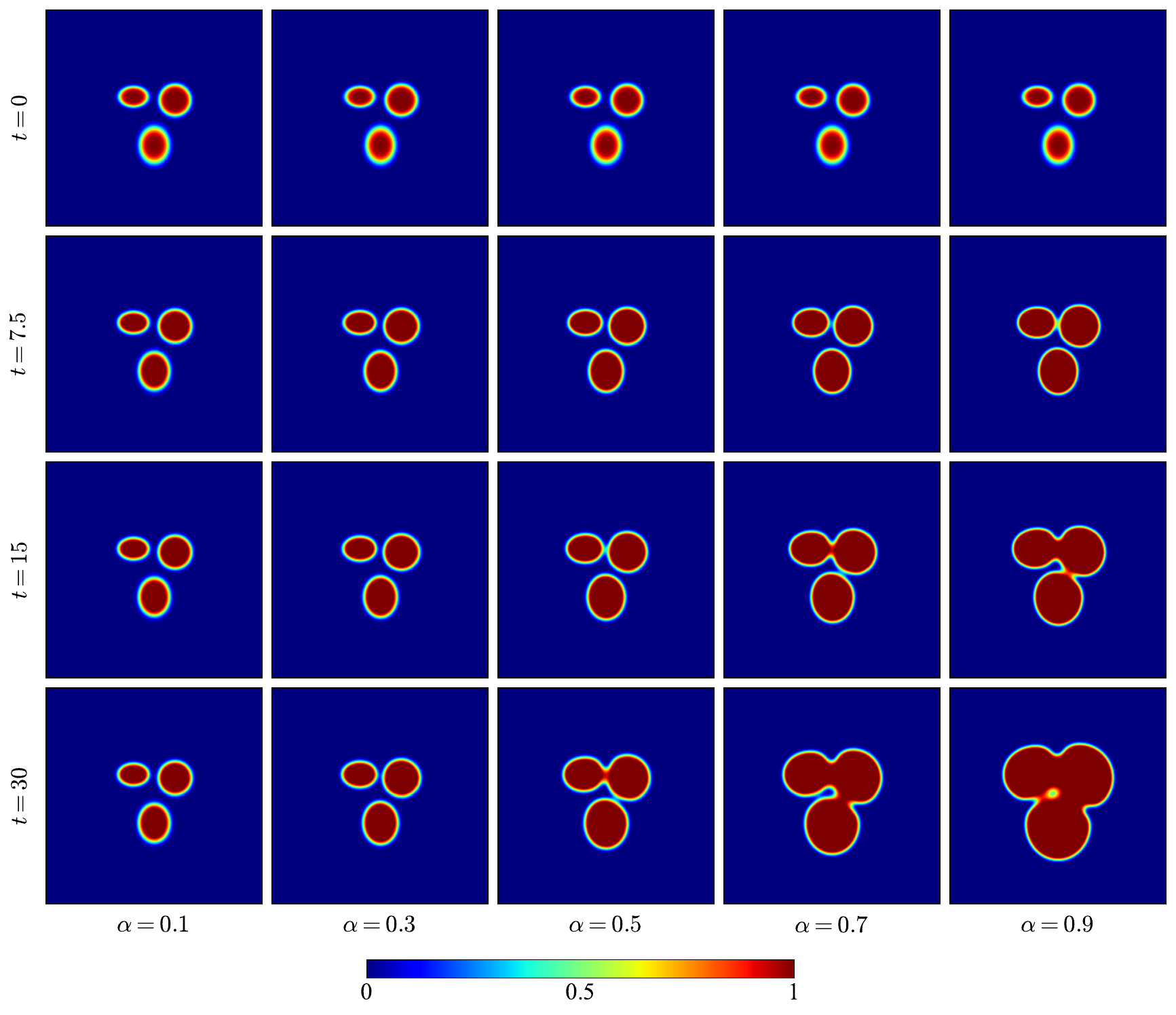}
\caption{Snapshots of the tumor phase $\phi$ for different fractional orders.}\label{fig:example2-phase}
\end{figure}

Figure~\ref{fig:example2-phase} shows the expansion and coalescence of three initially separated tumor regions.
Over the simulated time interval, larger fractional orders lead to more pronounced expansion and earlier coalescence.
At $t=30$, the three regions remain separate for $\alpha=0.1,0.3$, whereas the upper pair has merged for $\alpha=0.5$.
For $\alpha=0.7,0.9$, the upper pair has already merged by $t=15$, and all three regions form a connected tumor region by $t=30$.
The merged tumor region retains a lobed morphology reflecting the initial spatial arrangement.

\begin{figure}[!htbp]
\centering
\begin{subfigure}[t]{0.32\textwidth}
\centering
\includegraphics[width=\linewidth]{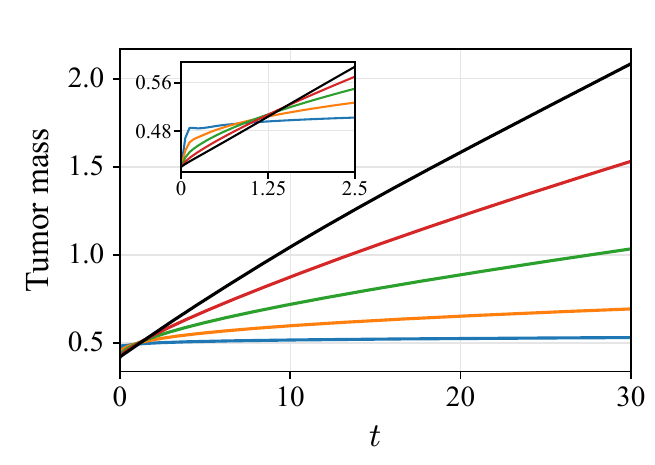}
\caption{}\label{fig:example2-tumor-mass}
\end{subfigure}
\hfill
\begin{subfigure}[t]{0.32\textwidth}
\centering
\includegraphics[width=\linewidth]{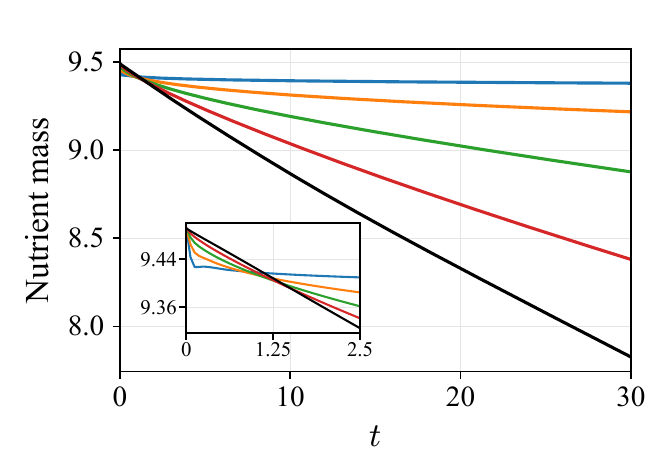}
\caption{}\label{fig:example2-nutrient-mass}
\end{subfigure}
\hfill
\begin{subfigure}[t]{0.32\textwidth}
\centering
\includegraphics[width=\linewidth]{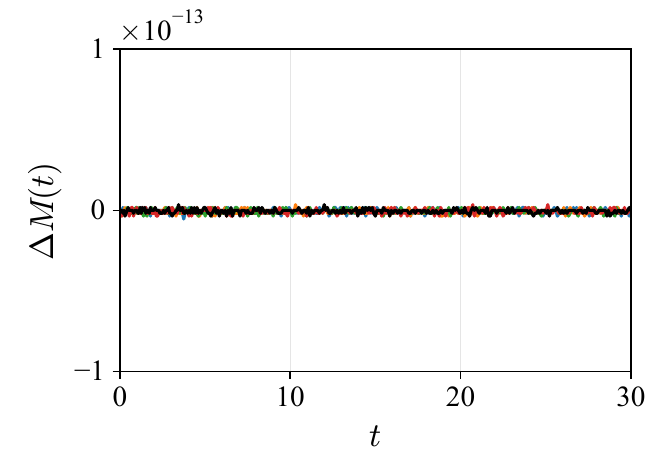}
\caption{}\label{fig:example2-total-mass}
\end{subfigure}
\par\medskip
\begin{subfigure}[t]{0.32\textwidth}
\centering
\includegraphics[width=\linewidth]{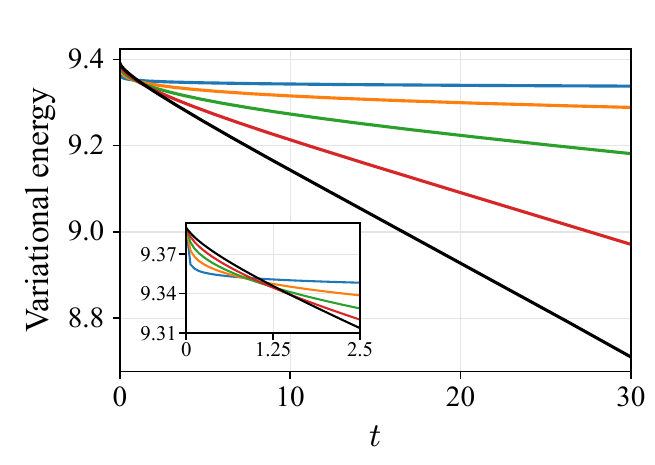}
\caption{}\label{fig:example2-energy-decay}
\end{subfigure}
\hfill
\begin{subfigure}[t]{0.32\textwidth}
\centering
\includegraphics[width=\linewidth]{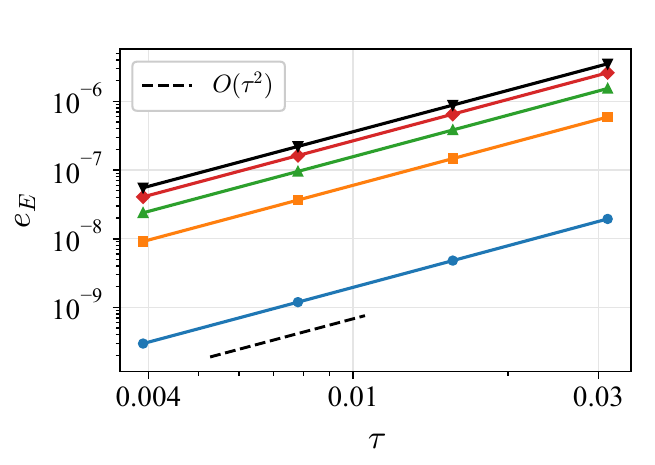}
\caption{}\label{fig:example2-energy-compatibility}
\end{subfigure}
\hfill
\begin{subfigure}[t]{0.32\textwidth}
\centering
\includegraphics[width=\linewidth]{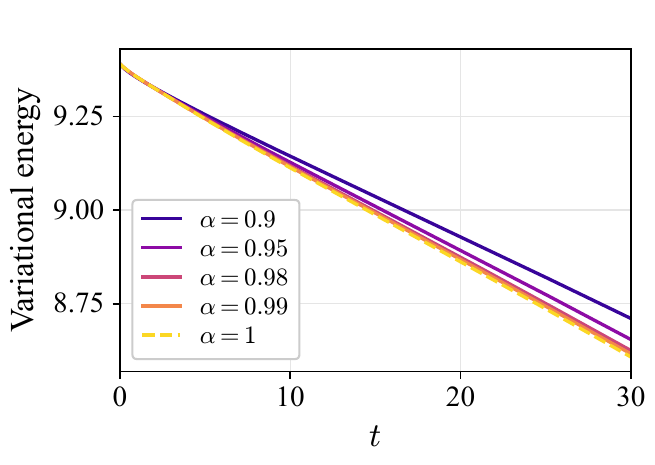}
\caption{}\label{fig:example2-integer-limit}
\end{subfigure}
\par\smallskip
\includegraphics[width=0.80\textwidth]{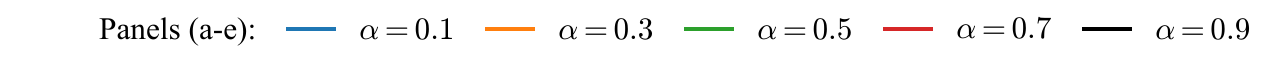}
\caption{Energy and mass evolution:
(a) tumor mass;
(b) nutrient mass;
(c) total-mass variation;
(d) variational energy;
(e) absolute difference between the discrete relaxation and original energies;
(f) variational energy as $\alpha\to1^-$.}
\label{fig:example2-diagnostics}
\end{figure}

Figure~\ref{fig:example2-diagnostics} reports the mass and energy evolution for Example~\ref{example2}.
Subfigures~\ref{fig:example2-diagnostics}(\subref{fig:example2-tumor-mass})--\ref{fig:example2-diagnostics}(\subref{fig:example2-total-mass}) show increasing tumor mass and decreasing nutrient mass,
with total mass conserved throughout.
Smaller $\alpha$ produces a faster initial increase in tumor mass followed by a pronounced slowdown,
whereas larger $\alpha$ yields slower initial growth but more sustained growth at later times,
resulting in greater final tumor mass.
Regarding energy dissipation,
the discrete variational energy decreases monotonically for all tested orders:
smaller $\alpha$ produces a rapid initial decrease,
after which the energy levels off,
whereas larger $\alpha$ sustains the decrease longer and yields a lower final energy;
see subfigure~\ref{fig:example2-diagnostics}(\subref{fig:example2-energy-decay}).
Subfigure~\ref{fig:example2-diagnostics}(\subref{fig:example2-energy-compatibility}) plots the energy difference $e_E(N):=|\widetilde E_{dm}^{N}-E^N|$ at $T=1$ on uniform grids.
For all five orders,
this difference is larger for larger $\alpha$ at fixed $\tau$ and decays as $O(\tau^2)$ under refinement,
supporting compatibility of the two discrete energies.
Finally, subfigure~\ref{fig:example2-diagnostics}(\subref{fig:example2-integer-limit}) compares the variational energy for $\alpha=0.9,0.95,0.98,0.99$ with the integer-order reference on the same time grid.
The differences are more noticeable at later times and diminish as $\alpha\to1^-$,
consistent with Remark~\ref{rem:compatibility_DMLR}.

\begin{example}[Comparison with experimental tumor spheroid growth]\label{example:u87}
We simulate in vitro U87 tumor spheroid growth and compare the results with the microscopy images for 5000 initially seeded cells in Figure 4 of~\cite{bach2024rapid}, with $1\,\mathrm{mm}$ and $1\,\mathrm{day}$ as the length and time scales, respectively.
We retain the phase evolution equation and use the quasi-static nutrient equation under spherical symmetry, with $\ell$ denoting the radial coordinate, on $0<\ell<0.65$, with $\sigma=1$ and $\partial_\ell\phi=\partial_\ell\mu_\phi=0$ at $\ell=0.65$, and $\partial_\ell\phi=\partial_\ell\mu_\phi=\partial_\ell\sigma=0$ at $\ell=0$.
The initial phase has the form
\[
\phi_0(\ell)=\frac12\left[1-\tanh\left(\frac{\ell-\ell_0}{2\sqrt2\,\epsilon_\phi}\right)\right],
\]
where $\ell_0$ and the initial day are specified separately for the microscopy and volume comparisons below.
We set $\epsilon_\phi=0.003$, $\delta=10^{-4}$, $\chi_0=0.05$, $\kappa_\phi=1$, $\lambda=0.95$, $M_\phi=0.1$, $M_\sigma=2\times10^{-6}$, and $P(\phi)=P_0\delta\phi^2$,
giving $D_\sigma=M_\sigma/\delta=0.02$.
We use a finite-volume discretization on 640 uniform spherical shells.
The time grid is $t_n=T(n/N)^\gamma$, where $\gamma=1$ for $\alpha=1$ and $\gamma=1.2$ for $\alpha=0.8$.
The computed radius is defined by the outer $\phi=0.5$ interface.
\end{example}

\begin{figure}[htbp]
\centering
\includegraphics[width=0.8\textwidth]{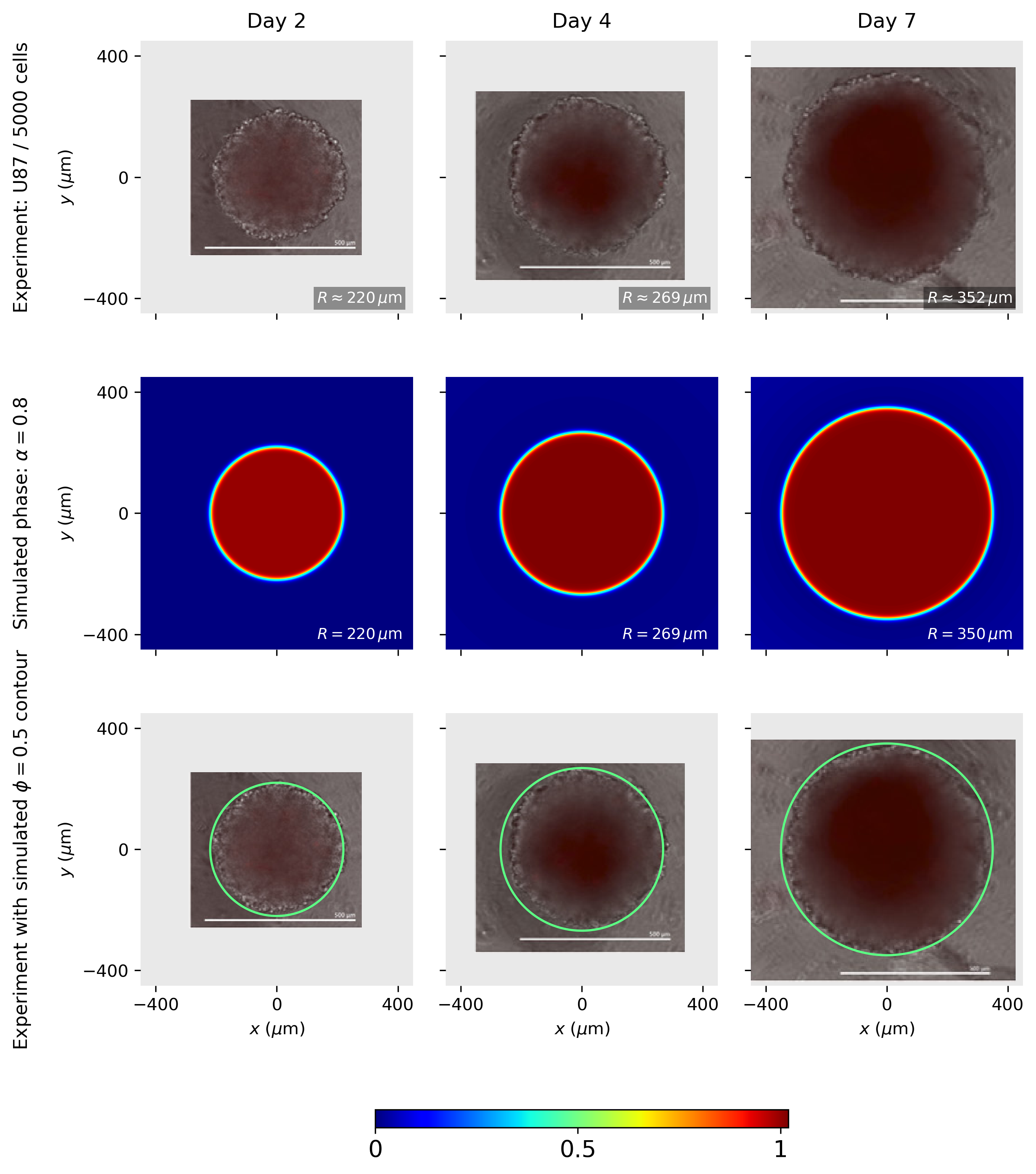}
\caption{U87 spheroid growth at days 2, 4, and 7:
microscopy images adapted from Figure 4 of Bach et al.~\cite{bach2024rapid} (top),
computed phase fields for $\alpha=0.8$ (middle),
and microscopy images overlaid with the corresponding $\phi=0.5$ contours (bottom).}
\label{fig:example3-phase-comparison}
\end{figure}

Figure~\ref{fig:example3-phase-comparison} compares the simulated spheroid growth for $\alpha=0.8$ with the microscopy images,
with the computed $\phi=0.5$ contours superimposed on the images for direct comparison.
The simulation starts at day 2 with $\ell_0=0.220$ estimated from the corresponding image,
and uses $T=5$ with $N=9600$ time steps.
We calibrate $P_0=0.66609$ to the approximate day-4 radius of $268.6\,\mu\mathrm m$.
At day 7,
the computed radius is approximately $349.7\,\mu\mathrm m$,
compared with the image-based estimate of $352.3\,\mu\mathrm m$,
giving an absolute relative error of $0.7\%$.

\begin{figure}[htbp]
\centering
\includegraphics[width=0.9\textwidth]{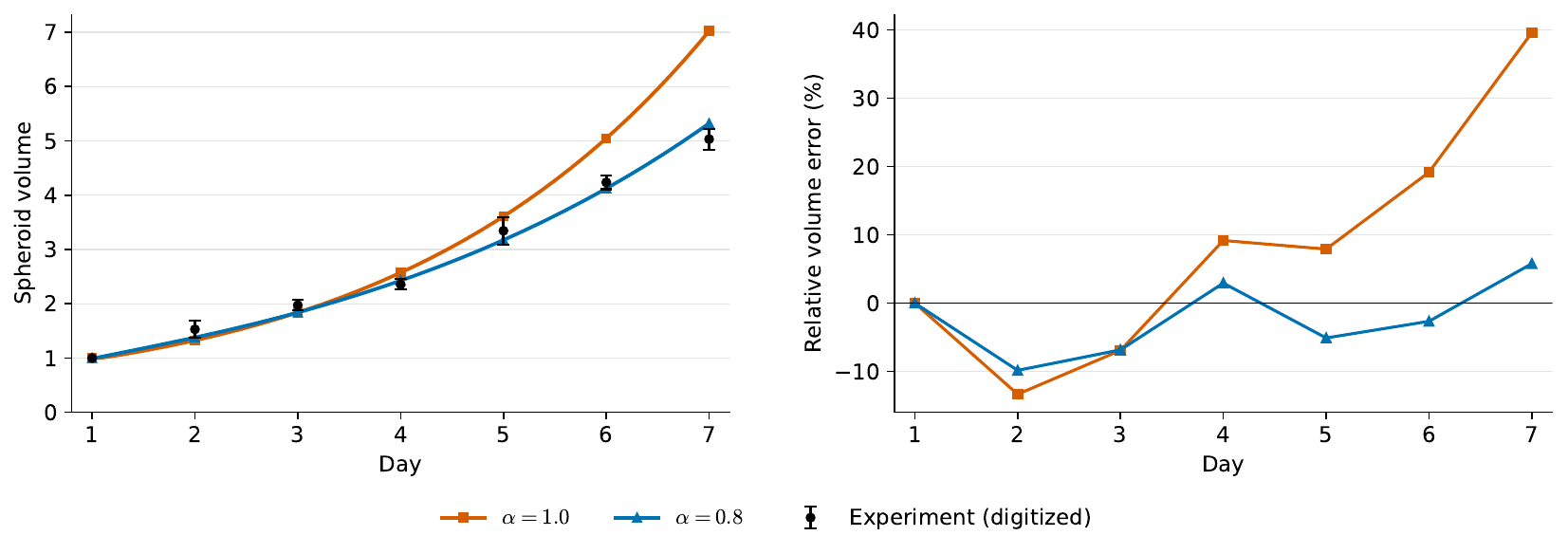}
\caption{Comparison with the 5000-cell spheroid volume data in Figure 2 of~\cite{bach2024rapid}, whose means and standard deviations are approximately digitized from the reported curves and error bars:
spheroid volume (left) and relative volume errors (right) for $\alpha=1$ and $0.8$.}
\label{fig:example3-volume-comparison}
\end{figure}

For Figure~\ref{fig:example3-volume-comparison},
we initialize the simulation and fractional memory at day 1 using $V_0\approx0.034 \,\mathrm{mm}^3$ estimated from Figure 1 of~\cite{bach2024rapid},
with $ \ell_0=[3V_0/(4\pi)]^{1/3}/(1\,\mathrm{mm})$,
retaining the calibrated parameters and using $T=6$ with $N=11520$.
Starting from the same day-1 volume,
the $\alpha=0.8$ simulation remains closer to the experimental mean volume than the integer-order simulation at every subsequent observation time through day 7.
This improvement is also reflected in the mean absolute percentage error over days 2--7,
which decreases from $16.0\%$ for $\alpha=1$ to $5.5\%$ for $\alpha=0.8$.

\begin{example}[Tumor growth with sustained nutrient supply]\label{example4}
This example is motivated by the simulations of unconfined tumor growth with a quasi-static nutrient field reported in~\cite{zieger2026phase}. 
We consider the domain $\Omega=[-1,1]^2$ and impose the boundary conditions
$\sigma = 1$ and $\partial_{\bm n}\phi = \partial_{\bm n}\mu_\phi = 0$ on $\partial\Omega$.
The parameters are set to $\epsilon_\phi=0.003$,
$\chi_0=0.05$,
$\delta=10^{-4}$,
$P_0=0.01$,
$M_\phi=0.3$,
and $M_\sigma=10^{-8}$,
with $P(\phi)=P_0\delta\phi^2$.
The initial tumor phase is defined by
\[
\phi_0(x,y) =\frac12\left[1-\tanh\left(\frac{\ell-\ell_0(\vartheta)}{2\sqrt2\,\epsilon_\phi}\right)\right],\quad
\ell_0(\vartheta) =0.55[1-0.005\cos(2\vartheta)-0.03\cos(4\vartheta)],
\]
where $\ell=\sqrt{x^2+y^2}$ and $\vartheta=\operatorname{atan2}(x,y)$.
The initial nutrient field is obtained from the quasi-static nutrient equation with $\phi=\phi_0$.
We use a $768\times768$ cell-centered finite-volume grid and 600 uniform time steps,
with $\tau_*=0.1$, $\kappa_\phi=1$, and $\lambda=0.95$.
\end{example}

\begin{figure}[!htb]
\centering
\includegraphics[width=0.79\textwidth]{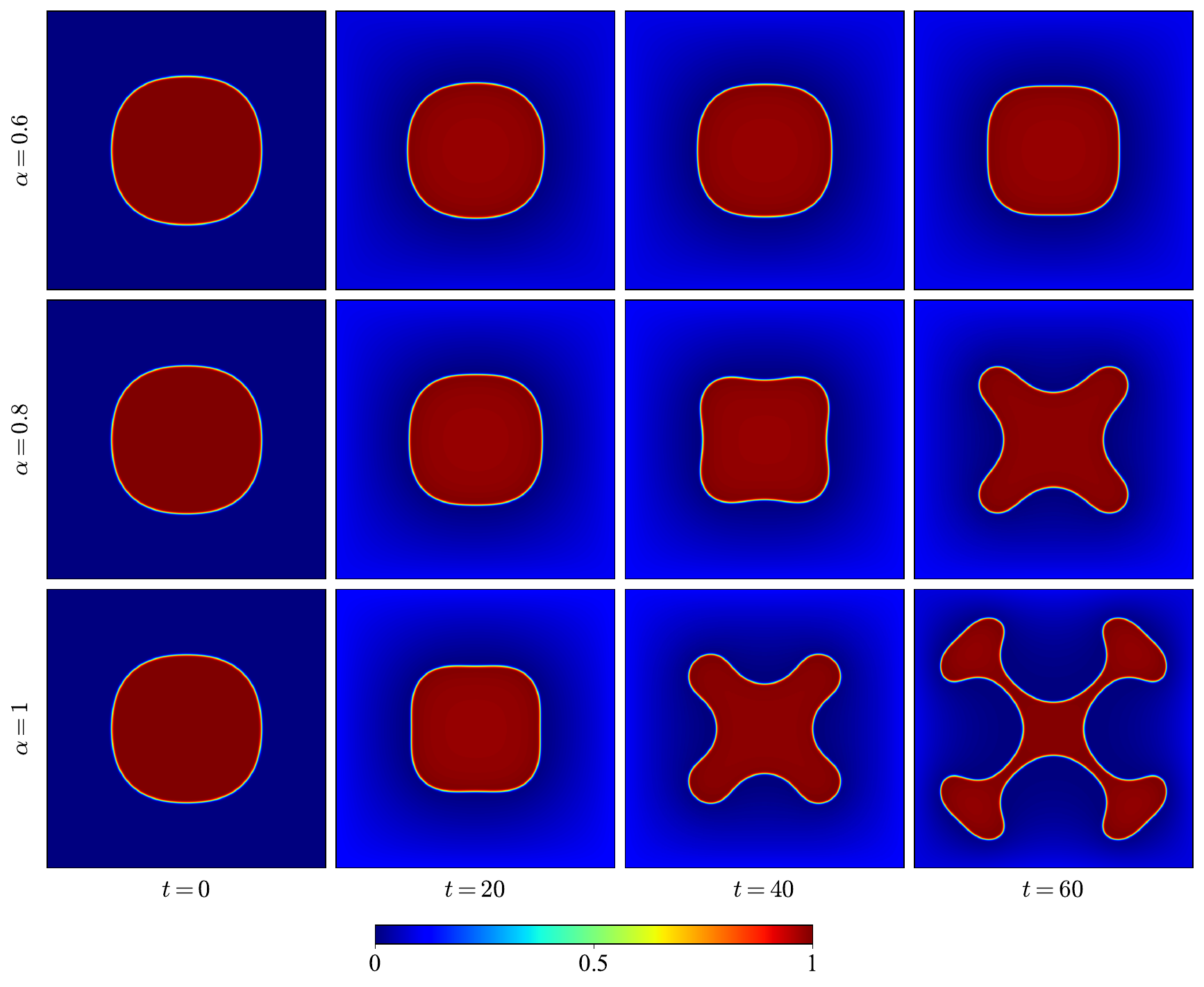}
\caption{Tumor phase evolution under sustained nutrient supply for $\alpha=0.6,0.8,1$.}\label{fig:example4-phase}
\end{figure}

Figure~\ref{fig:example4-phase} illustrates the effect of $\alpha$ on tumor evolution under sustained nutrient supply.
For $\alpha=1$,
the rounded initial shape develops into a four-armed structure by $t=40$,
with arms growing from the corners along the diagonals;
by $t=60$,
the arms have elongated further and broadened at their tips.
For $\alpha=0.8$,
the same deformation develops more slowly,
and the shape at $t=60$ resembles that for $\alpha=1$ at $t=40$,
whereas for $\alpha=0.6$ no arms form and the outline only flattens into a rounded square.
These results indicate that smaller values of $\alpha$ delay the amplification of interfacial perturbations and the development of arm-like morphologies over the simulated interval.

\section{Conclusions}\label{sec:conclusion}
This work established a variational energy framework for a mass-conservative subdiffusive tumor--nutrient model incorporating chemotactic coupling and memory-dependent reaction exchange.
The weighted Riemann--Liouville identities provided the tools to construct variational energies and establish the corresponding dissipation laws for the continuous model and its linear relaxation approximations.
For numerical approximation, we developed linear relaxation, damped linear relaxation, and damped modified linear relaxation schemes.
Temporal damping limits the influence of alternating auxiliary-variable errors on the state variables but leaves the underlying recurrence unchanged.
The a posteriori relaxation correction further suppresses the alternating auxiliary-variable error mode while preserving the discrete dissipation law and the linear structure of the scheme.
All three schemes conserve total mass and satisfy discrete variational energy dissipation laws.
Both the schemes and these laws are asymptotically compatible with their integer-order counterparts.
Numerical simulations show that larger fractional orders lead to more sustained tumor growth and more pronounced morphological evolution over the simulated intervals, while an appropriate fractional order improves agreement with experimental spheroid growth data compared with the integer-order model.

The present analysis establishes structural properties of the schemes, while their convergence behavior is assessed numerically.
A rigorous error analysis accounting for initial temporal singularities, nonuniform time steps, and the alternating auxiliary-variable error remains an important direction for future work.

\bibliographystyle{siamplain}
\bibliography{FractionalTumor}
\end{document}